\documentclass[12pt,twoside]{article}
\usepackage{hyperref}
\usepackage{geometry}
\usepackage{xcolor}
\usepackage{fancyhdr}
\usepackage{mathdots}
\usepackage{bbold}

\usepackage[labelfont=bf]{caption}
\usepackage{graphicx}
\usepackage{amsmath}
\usepackage{amssymb}
\usepackage{stmaryrd}
\usepackage{eqparbox}
\usepackage{float}
\usepackage{booktabs}
\usepackage{tabularx}
\usepackage{titling}
\usepackage{tikz-cd}
\usepackage{amsthm}
\usepackage{fancyhdr}
\usepackage{lastpage}
\usepackage{titlesec}
\usepackage[T1]{fontenc}
\titlespacing*{\section}
{0pt}{0.3ex plus 1ex minus .2ex}{0.3ex plus .2ex}
\titlespacing*{\subsection}
{0pt}{0.3ex plus 1ex minus .2ex}{0.3ex plus .2ex}
\titlespacing*{\subsubsection}
{0pt}{0.3ex plus 1ex minus .2ex}{0.3ex plus .2ex}
\makeatletter
\newcommand{\extp}{\@ifnextchar^\@extp{\@extp^{\,}}}
\def\@extp^#1{\mathop{\bigwedge\nolimits^{\!#1}}}
\makeatother
\newtheorem{prop}{Proposition}[section]
\newtheorem{Coro}{Corollary}[section]
\newtheorem{Lem}{Lemma}[section]
\newtheorem{Thm}{Theorem}[section]
\newtheorem{Def}{Definition}[section]
\newtheorem{Example}{Example}[section]

\newcommand\frontmatter{%
    \cleardoublepage
  \pagenumbering{roman}}

\newcommand\mainmatter{%
    \cleardoublepage
  \pagenumbering{arabic}}

\newcommand\backmatter{%
  \if@openright
    \cleardoublepage
  \else
    \clearpage
  \fi
   }
\date{}

\title{\normalsize Para-Holomorphic Algebroids and Para-Complex Connections} 

\author{by\\  \normalsize Aidan Patterson
\\ \\  \\   \normalsize A thesis\\ \normalsize presented to the University Of Waterloo\\ \normalsize in fulfilment of the \\ \normalsize thesis requirement for the degree of\\ \normalsize Master of Mathematics \\ \normalsize in \\ \normalsize  Pure Mathematics \\ \\ \\ \\ \\ \\ \normalsize Waterloo, Ontario, Canada, 2021 \\ \normalsize \copyright \, \, Aidan Patterson 2021}
\usepackage{lipsum}
\renewcommand\footnotemark{}
\begin{document}
\frontmatter{}
\maketitle
\thispagestyle{empty}
\vspace{-9em}
\pagebreak
\section*{Author's Declaration}
I hereby declare that I am the sole author of this thesis. This is a true copy of the thesis, including any required final revisions, as accepted by my examiners. I understand that my thesis may be made electronically available to the public.
\pagebreak
\begin{abstract}
  \noindent The goal of this paper is to develop the theory of Courant algebroids with integrable para-Hermitian vector bundle structures by invoking the theory of Lie bialgebroids. We consider the case where the underlying manifold has an almost para-complex structure, and use this to define a notion of para-holomorphic algebroid. We investigate connections on para-holomorphic algebroids and determine an appropriate sense in which they can be para-complex. Finally, we show through a series of examples how the theory of exact para-holomorphic algebroids with a para-complex connection is a generalization of both para-K\"{a}hler geometry and the theory of Poisson-Lie groups.
  \end{abstract}
\pagebreak
\section*{Acknowledgements}
Thank you to Professor Moraru for supervising this thesis, and being a valuable mentor for much of my undergraduate and all of my graduate degree. I would certainly not be here without you. Thank you to my two additional referees, Professor Webster and Professor Hu, who were incredibly accommodating and agreed to do this on such short notice. Thank you to professor Harada and Professor Lane at McMaster University who first introduced me to the Iwasaw decomposition. Thank you to the staff at the University of Waterloo who helped me along the way, and particularly Rose in the Math Coffee \& Doughnut shop. Thank you to my friends in CPSO, CPSHR, AB, and my other Friends of the Filipino People in Struggle who helped me balance my life with my schoolwork during the writing of this document. Finally, Thank you to all my friends and family who have been a constant source of support along the way. 
\pagebreak
\tableofcontents
\pagebreak
\mainmatter{}
\section{ \normalsize Introduction}
    \indent There has been a growing interest in generalized geometry since Hitchin introduced the concept in his paper \textit{Generalized Calabi-Yau Manifolds} in 2002 \cite{Hitchin1}. Given a smooth manifold $M$, the natural setting for generalized geometry is the vector bundle $\mathbb{T}M:= TM\oplus T^{*}M$. This vector bundle comes equipped with a natural anti-symmetric bracket, $[\cdot,\cdot]$, introduced by Courant in his 1990 paper \textit{Dirac Manifolds} \cite{Courant}, a symmetric bilinear form $\langle \cdot,\cdot \rangle$ and a natural anchor map $\rho: TM\oplus T^{*}M\rightarrow TM$ given by projection onto the first factor. For $X\oplus \xi, Y\oplus \eta \in \Gamma(TM\oplus T^{*}M)$, the bracket and symmetric pairing are given explicitly by:
    \begin{align}\label{Standard Algebroid}
        [X\oplus \xi , Y\oplus\eta] &= [X,Y] \oplus \mathcal{L}_X\eta - \mathcal{L}_Y \xi -\frac{1}{2}d(\iota_{X}\eta - \iota_{Y}\xi),
        \\
       \nonumber \langle X \oplus \xi , Y\oplus \eta \rangle &= \xi(Y) + \eta(X).
    \end{align}
    \indent Courant also introduced the notion of a Dirac subbundle in \cite{Courant}.
    \begin{Def}
    Let $M$ be a smooth manifold. A Dirac structure on $\mathbb{T}M$ is a subbundle $L\subset \mathbb{T}M$ that is maximally isotropic and whose space of sections, $\Gamma(L)$, is  closed under the Courant bracket. 
    \end{Def}
    \noindent Dirac structures correspond directly to Poisson bivectors in the case where $L\neq TM$ or $T^{*}M$ (for a good explanation of why this is the case, see \cite{Meinrenken}), and so the search for Dirac subbundles is equivalent to the search for Poisson bivectors on $M$. 
    \\
    
    \indent On the other hand, research in para-complex geometry has been expanding as well. The fundamental object here is the \textit{almost product manifold}.
    \begin{Def}
    Let $M$ be a smooth manifold, and $J\in \Gamma(End(TM))$ with $J^2=Id_{TM}$. The pair $(M,J)$ is called an almost product manifold.
    \end{Def}
    \noindent Since $J^2 = Id_{TM}$, $J$ has two eigenvalues, $\pm1$, and so $TM$ admits a decomposition into the direct sum of the $\pm1$-eigenbundles $TM = T^{+}M\oplus T^{-}M$. If the eigenbundles have the same rank, then $J$ is called an \textit{almost para-complex structure}. The pair $(M,J)$ is called a \textit{half integrable para-complex manifold} if one of the eigenbundles is closed under the Lie bracket, and a \textit{para-complex manifold} if both eigenbundles are integrable. For an interesting survey of the historical context, and some basic results in the field leading up to the mid 1990's, see \cite{Cruceanu}. Just as in the complex case, one can ask what it means for an almost para-complex structure to be compatible with a metric on $M$. This gives the following definition.
    \begin{Def}
    Let $(M,J)$ be an almost para-complex manifold and $g$ be a non-degenerate symmetric bilinear form on $M$. Then we say that $J$ is \textit{compatible} with $g$ if $g(J\cdot,J\cdot) = -g(\cdot,\cdot)$ and in this case, we refer to $(M,g,J)$ as an \textit{almost para-Hermitian manifold}. 
    \end{Def}
    \indent It was Bejan who introduced the concepts of para-complex and para-Hermitian vector bundles in her paper \textit{The Existence Problem of Hyperbolic Structures on Vector Bundles} \cite{Bejan}. These structures lack a concept of integrability however, and so these constructions are done basically at the level of vector spaces. Importantly, Erdem introduced the concept of a para-holomorphic map between para-complex manifolds. In the context of vector bundles, we understand para-holomorphic maps as vector bundle morphisms $T:(E,J_E)\rightarrow (F,J_F)$ such that $J_F\circ T = T \circ J_{E}$. Erdem also introduced the idea of ``para-complexifying" a vector bundle using the para-complex numbers $C$ to obtain $\pm j$-eigenbundles, similar to the complex case, but with the caveat that $j^2 =1 $.
    \\
    
    \indent Where the concept of Courant algebroids and para-Hermitian vector bundles meet is in Svoboda's paper \textit{Algebroid Structures on Para-Hermitian Manifolds} \cite{Svoboda}. Motivated by the physics of Double Field Theory, Svoboda identifies para-Hermitian structures on Courant algebroids as corresponding to Lie bialgebroids, which are thoroughly studied in \cite{Roytenberg}. The condition of integrability of the eigenbundles of the para-complex structure becomes closure under the Courant bracket, and so one obtains a pair of transverse Dirac structures corresponding to the para-Hermitian structure.
    \\
    
    \indent In this paper, we are interested in exploring para-Hermitian algebroids $E\rightarrow M$ over para-Hermitian manifolds $(M,g,J)$ in order to show that an even wider class of of objects can be realized in the para-complex setting than previously thought. In this setting, it is natural to ask questions like: ``Is the obstruction to integrability of the eigenbundles of $E$ related to structures on $TM$?", ``What happens if the anchor map $\rho:E\rightarrow TM$ is para-holomorphic?", and ``Given a connection $A:TM\rightarrow E$, what does it mean for $A$ to be para-holomorphic?". The contribution that this paper aims to make is to introduce the concept of para-holomorphic algebroids, and para-complex connections. The interplay between para-complex connections and para-Hermitian/holomorphic algebroids is particularly interesting. For instance, we find that 
    \begin{Thm}
    If $(E,J)$ is an exact para-Hermitian algebroid over a para-Hermitian manifold, and admits a flat para-complex connection $A$, then $E$ is a para-holomorphic algebroid.  
    \end{Thm}
    \indent A consequence of considering para-complex connections is that we arrive naturally at the concept of a \textit{split-para-complex structure}. In short, a split-para-complex structure is a vector bundle $E$ together with two para-complex structures $J,K$, such that $JK=KJ$. For an exact para-Hermitian algebroid $(E,J)$ with a connection $A$ and an anchor map $\rho : E\rightarrow TM$, we can define the para-complex structure $K$ to take the value $+1$ on $A(TM)$ and $-1$ on $\rho^{*}(T^{*}M)$. In general, we can take any splitting $E=\rho^{*}(T^{*}M)\oplus H$, and define the para-complex structure $K$ similarly. This identification of connections and almost para-complex algebroids leads us to the observation: 
    \begin{Thm}
    If $(E,J)$ is an exact para-holomorphic algebroid, then $E$ admits a flat para-complex connection if and only if $(E,J,K)$ is split-para-complex.   
    \end{Thm}
    \indent Finally, we will present a class of para-Hermitian algebroids over the exact Courant algebroid $\mathbb{T}G$ for a Lie group $G$ with a quadratic Lie algebra $\mathfrak{g}$, as well as a class of exact para-holomorphic algebroids with flat para-complex connections over $\mathbb{T}G$ when $\mathfrak{g}$ is also semi-simple. In the second class of examples, we realize a compact real form of the Lie algebra $\mathfrak{g}$ as corresponding to a special para-holomorphic structure on $\mathbb{T}G$, where $G$ is the Drinfeld double of $K$, endowed with a compatible para-complex connection induced by the diagonal inclusion $G \hookrightarrow G\times G$ (referred to as the Cartan-Dirac structure). We will also construct a separate para-holomorphic structure with flat para-complex connection on $\mathbb{T}G$ in the case where $G$ is the Drinfeld double of a complete simply-connected Poisson-Lie group $(K,\Pi_K)$. The goal of this section is to concretely establish that exact para-holomorphic algebroids with flat para-complex connections are not only common, but comprise some well known objects. The impression we try to give is that para-Hermitian geometry is general enough to include interesting objects and provides a direction for future generalizations. 
    \\
    
    \indent Roytenberg studied Lie bialgebroids using the framework of supermanifolds in \cite{Roytenberg}, and so in the future it would be interesting to generalize para-holomorphic algebroids using this formalism. In particular, it would be interesting to try and cast the para-complex structure as a function on $T^{*}\Pi E_{+}$, as they do with other important objects like the Courant bracket. It would also be interesting to try and find examples of para-holomorphic structures with compatible para-complex connections that are not simply the diagonal embedding, and are perhaps not flat as well. 
\pagebreak
\section{ \normalsize Para-Complex Vector Bundles}
\indent Para-complex manifolds and para-complex vector bundles are straightforward generalizations of complex manifolds and vector bundles. The utility of the para-complex setting is that the splitting of the vector bundle into eigenbundles occurs at the level of the vector bundle, and does not require you to ``para-complexify". One can generalize basically any fact about complex manifolds and vector bundles to the para-complex setting to obtain analogous results and fundamental objects, which is what we will summarize in this chapter. 
\\

\indent Of special importance is the so called para-Hermitian structure. The main interest in this structure comes from the fact that the eigenbundles of the para-complex structure become maximally isotropic. In the context of quadratic Lie algebras, orthogonal maximally isotropic subalgebras arise naturally in the study of Poisson-Lie groups. With our goal of finding a single object that corresponds naturally to para-K\"{a}hler manifolds and also Poisson-Lie groups, we present the theory of para-complex vector bundles. 
\subsection{ \normalsize Basic Properties}
Throughout this paper, let $M$ be a smooth manifold of dimension $n$ and $E\rightarrow M$ be a smooth real vector bundle over $M$. 
\begin{Def}
A product structure on $E$ is a section $J\in \Gamma(End(E))$ such that $J^2= Id_{E}$.
\end{Def}
\indent On each fiber $E_p := \pi^{-1}(p)$, $p\in M$, we find that $J$ has the two possible eigenvalues $1$ and $-1$. This induces a decomposition of $E_p$ into the direct sum of the eigenspaces $E_p = E_p^{+}\oplus E_{p}^{-}$, which in turn induces an eigenbundle decomposition $E= E^{+} \oplus E^{-}$. When the subbundles $E^{+}$ and $E^{-}$ have the same rank (hence the rank of $E$ must be even), we say that $J$ is a \textit{para-complex structure}. This decomposition of the vector bundle $E$ induces a decomposition of the dual bundle as follows:
\\

Let $E_{+}^{*} := \{\omega \in E^{*} \, \vert \,  \omega(Z)=0 \textrm{ for all } Z\in E^{-}\}$ and  $E_{-}^{*} := \{\omega \in E^{*} \, \vert \,  \omega(Z)=0 \textrm{ for all } Z\in E^{+}\}$. We have the decomposition $E^{*} = E^{*}_{+}\oplus E^{*}_{-}$. This splitting induces a decomposition of the $k^{th}$-exterior bundle of $E$ given by 
\begin{align*}
    \extp^k E^{*} &= \bigoplus_{p+q=k} E^{*}_{(p,q)},
\end{align*}
where in this case $E^{*}_{(p,q)} := \Lambda^p E^{*}_{+}\wedge \Lambda^q E_{-}^{*}$. An alternative definition involves the ring of para-complex numbers $C = \{ a+ jb\,  \vert \, j^2 =1, \,   \,  a,b\in\mathbb{R}\}$.
\begin{Def}
    A para-complex vector bundle is a vector bundle whose fibers are isomorphic to the free module $C^n$, with the multiplication by $j$ as the para-complex structure. 
\end{Def}
The transition functions of such a bundle will then be smooth (not necessarily para-holomorphic) maps $g_{\alpha\beta} : (U_\alpha\cap U_\beta)\times C^n \rightarrow (U_\alpha\cap U_\beta)\times C^n$ satisfying the relations $g_{\alpha\alpha} = Id_{E\vert_{U_{\alpha}}}$ and $g_{\alpha\beta}\circ g_{\beta \gamma}\circ g_{\gamma \alpha}= Id_{E\vert_{U_{\alpha}\cap U_{\beta}\cap U_{\gamma}}}$. Just as in the complex case, placing restrictions on the transition functions allows us to introduce the concept of para-holomorphic vector bundles, which we discuss later. We now consider the case of the tangent bundle and the interaction of these eigenbundles with the natural Lie algebroid structure on $TM$ that comes from the Lie bracket.
\subsection{ \normalsize The Case of the Tangent Bundle}
An important case to consider is when $E=TM$. In this case, the pair $(M,J)$ is referred to as an \textit{almost para-complex manifold}. The famous Newlander-Nirenberg theorem in this context relates the almost para-complex structure $J$ on $TM$ to the integrability of the eigenbundles $T^{+}M$ and $T^{-}M$.
\begin{Thm}
(Newlander-Nirenberg \cite{Moroianu}): Let $(M,J)$ be an almost para-complex manifold. Then $J$ is induced by a para-holomorphic atlas if and only if the distributions $T^{+}M$ and $T^{-}M$ are integrable.
\end{Thm}
\indent We will discuss the concept of para-holomorphicity in the next section. The integrability of the eigenbundles induces coordinates on $M$ of the form $(\mathbf{z}_{+},\mathbf{z}_-)$. The partial derivatives with respect to the $\mathbf{z}_{\pm}$ coordinates span eigenbundles $T^{\pm}M$. As in the complex case, there is a $(2,1)-$tensor $N^J$, called the \textit{Nijenhuis tensor}, that measures the integrability of these eigenbundles:
\begin{align}\label{Nijenhuis Tensor}
   4 N^J(X,Y) = \left([X,Y] - J[X,JY] - J[JX,Y] +[JX,JY]\right).
\end{align}
One can easily see that $N^J=0$ if and only if $T^{+}M$ and $T^{-}M$ are integrable. In the case where $M$ is itself a para-complex manifold (meaning that it has an integrable para-complex structure), we can define the notion of a para-holomorphic vector bundle over $M$ using the concept of $(p,q)$-forms with values in $E$. If $M$ is a para-complex manifold, then we have a decomposition of the differential forms on $M$ into types and can give an expression for the image of the deRham differential with respect to this decomposition:
\begin{align*}
    \Omega^k M &= \bigoplus_{p+q=k}\Omega^{(p,q)}M,
    \\
    d\Gamma(\Omega^{(p,q)}M)&\subset \Gamma(\Omega^{(p+1,q)}M)\oplus \Gamma(\Omega^{(p,q+1)}M).
\end{align*}
This is summed up in the following theorem.
\begin{Thm}\label{Classical Decomposition of Forms}
(\cite{Krahe}, \cite{Moroianu}) Let $J$ be an almost para-complex structure on $M$. Then the following are equivalent:
\begin{enumerate}
    \item J is a para-complex structure.
    \item $T^+ M$ and $T^{-}M$ are integrable. 
    \item $d\Gamma(\Omega^{(1,0)}M)\subset \Gamma(\Omega^{(2,0)}M\oplus \Omega^{(1,1)}M)$ and $d\Gamma(\Omega^{(0,1)}M)\subset \Gamma(\Omega^{(1,1)}M\oplus \Omega^{(0,2)}M)$.
    \item $d\Gamma(\Omega^{(p,q)}M)\subset \Gamma(\Omega^{(p+1,q)}M \oplus \Omega^{(p,q+1)}M)$.
    \item $N^J =0$.
\end{enumerate}
\end{Thm}
\indent The fourth condition allows us to decompose the exterior derivative into the sum of two operators, $d= \partial_+ + \partial_-$, with the property that $\partial_+ : \Omega^{(p,q)}M\rightarrow \Omega^{(p+1,q)}M$ and $\partial_-: \Omega^{(p,q)}M\rightarrow \Omega^{(p,q+1)}M$. From the fact that $d^2 =0$, we derive the identities $\partial_+^2 = 0$, $\partial_-^2 = 0$ and $\partial_-\partial_+ +\partial_+ \partial_-=0$ from the fact that the type decomposition of forms is disjoint. In keeping with the complex setting, we obtain a local $\partial_\pm$-Poincar\'{e} lemma and a local $\partial_+\partial_-$-lemma. 
\begin{Lem}
    (Local $\partial_\pm$-Poincar\'{e} Lemma \cite{Krahe}): Any $\partial_+$-closed form $\omega \in \Omega^{(p,q)}U$, $p\geq 1$, is locally $\partial_+$-exact, and similarly for $\partial_-$ when $q\geq 1$.
\end{Lem}
\begin{Lem}
    (Local $\partial_+\partial_-$-Lemma \cite{Krahe}): Let $\omega \in \Omega^{(1,1)}M$ be a 2-form. Then $\omega$ is closed if and only if for every point $x\in M$, there exists an open neighborhood $U$ containing $x$ such that $\omega\vert_U = \partial_+ \partial_- u$ for some real function $u$ on $U$.
\end{Lem}
\subsubsection{ \normalsize Para-Holomorphic Bundles}
\indent Let $(M,J)$ be a para-complex manifold and $f:M\rightarrow C$ be a smooth function. Following \cite{Erdem}, we say that $f$ is \textit{para-holomorphic} when $df\circ J = j\circ df$. We note that any function $f:M\rightarrow C$ can be expressed as $f = \frac{1+j}{2} f_1 + \frac{1-j}{2}f_2$ for some real-valued functions $f_1,f_2\in C^{\infty}(M)$. Now, since $M$ is a para-complex manifold, we have local coordinates $(\mathbf{z}_+, \mathbf{z}_-)$, and in these coordinates, the para-Cauchy-Riemann equations read
\begin{align}
    \partial_- f_1= 0, \, \, \, \, \, \, \, \, \, \,  \partial_+ f_2=0.
\end{align}
Alternatively, if we model $M$ locally as $C^n$, then using the $+j$-and $-j$-eigencoordinates of $j$ extended to $TM\otimes_\mathbb{R}C$, $\mathbf{z}$ and $\overline{\mathbf{z}}$, the para-Cauchy-Riemann equations take on the familiar form of $\overline{\partial}f=0$, where $\overline{\partial}$ corresponds to partial derivatives with respect to the $\overline{\mathbf{z}}$ coordinates. We desire a simple expression for para-holomorphicity similar to the one in the complex setting. In order to do this, we consider the operators 
\begin{align}\label{Dolbeault}
    \partial &= \frac{1+j}{2} \, \partial_{+} + \frac{1-j}{2} \, \partial_{-}, \, \, \, \, \, \, \, \, \, \, \, \, \, \, \, \, 
    \overline{\partial} = \frac{1+j}{2} \,\partial_- + \frac{1-j}{2} \, \partial_+ .
\end{align}
We can say that $f:M\rightarrow C$ is para-holomorphic if and only if $\overline{\partial}f  = 0 $, which is useful because it allows us to avoid para-complexifying the tangent bundle in order to obtain $\pm j$-eigenbundles. Further, we retain the identity $d = \partial + \overline{\partial}$. As was hinted earlier, $E\rightarrow M$ is a para-holomorphic vector bundle over a para-complex manifold if there exists an atlas with para-holomorphic transition functions. With this in mind, we can now examine what it means for the transition functions of a para-complex bundle $E$ over a para-complex manifold $M$ to be para-holomorphic. Given a trivialization $\Psi_{U_\alpha} : \pi^{-1}(U_\alpha)\rightarrow U_\alpha\times \mathbb{R}^{2n}$, we can construct the transition functions $ \Psi_{U_\alpha}\circ\Psi_{U_\beta}^{-1}(x,v) = (x,\tau_{\alpha\beta}(x)v) $. The component functions of the map $\tau_{\alpha\beta} : U_\alpha \cap U_\beta \rightarrow GL_{2n}(\mathbb{R})$ are of interest here when determining what it means for this map to be para-holomorphic. To understand what this means, we first consider that these transition functions should be compatible with the para-complex structure on $\mathbb{R}^{2n}$. If we choose eigenvectors as the basis, we can express $\tau_{\alpha\beta}$ as a block matrix in the $z_{+}$ and $z_{-}$ coordinates. we find that when
\begin{align*}
    J\circ \begin{bmatrix}
    \rho_{\alpha\beta}^{+} & \rho_{\alpha \beta}^{-}\\
    \sigma_{\alpha\beta}^+ & \sigma_{\alpha\beta}^{-}
    \end{bmatrix} &= \begin{bmatrix}
    \rho_{\alpha\beta}^{+} & \rho_{\alpha \beta}^{-}\\
    \sigma_{\alpha\beta}^+ & \sigma_{\alpha\beta}^{-}
    \end{bmatrix} \circ J,
    \\
    \begin{bmatrix}
    \rho_{\alpha\beta}^{+} & \rho_{\alpha \beta}^{-}\\
    -\sigma_{\alpha\beta}^+ & -\sigma_{\alpha\beta}^{-}
    \end{bmatrix} &= \begin{bmatrix}
    \rho_{\alpha\beta}^{+} & -\rho_{\alpha \beta}^{-}\\
    \sigma_{\alpha\beta}^+ & -\sigma_{\alpha\beta}^{-}
    \end{bmatrix}
\end{align*}
and so
\begin{align*}
    \tau_{\alpha\beta} &= \begin{bmatrix}
    \rho_{\alpha\beta} & 0 \\
    0 & \sigma_{\alpha\beta}
    \end{bmatrix}.
\end{align*}
Further, we will see that in light of Proposition \ref{real-para-complex} and the previous discussion on para-holomorphic functions, and the fact that under the decomposition $ E_{+}\oplus E_{-}$ the transition functions are in the form $\rho_{\alpha\beta}\oplus \sigma_{\alpha\beta}$, the transition functions of the para-complexified bundle would be $\frac{1+j}{2}\rho_{\alpha\beta}+ \frac{1-j}{2}\sigma_{\alpha\beta}$. For this reason, we present the following.
\begin{prop}
\cite{Bejan} Let $E\rightarrow M$ be a rank $2n$ real vector bundle over a para-complex manifold. Then $E$ is para-complex if and only if the structure sheaf (I.e. the collection of transition functions) admits a decomposition
\begin{align*}
    g_{\alpha\beta} &= \begin{bmatrix}
    \rho_{\alpha\beta} & 0 \\
    0 & \sigma_{\alpha\beta},
    \end{bmatrix}
\end{align*}
where $\rho_{\alpha\beta}$ and $\sigma_{\alpha\beta}$ both have rank $n$. Moreover $E$ is para-holomorphic if and only if $\partial_{-}(\rho_{\alpha\beta})_{ij} = 0$ and $\partial_{+}(\sigma_{\alpha\beta})_{ij} = 0$.
\end{prop}
\indent We now consider the bundle of forms. Beginning on the manifold $M$, we can construct the bundle $\Lambda^{(p,q)}E = \Lambda^{(p,q)}M\otimes E$ of $E$-valued forms on $M$ of type $(p,q)$. At this point, it is pertinent to address the difference between the real bundle $E\rightarrow M$ with the para-complex structure $J$ and its para-complexification. Recall that the para-complexification of the bundle $E$ is simply the vector bundle $E_C$ whose fibers are $(E_C)_x = E_x\otimes_{\mathbb{R}} C$, $x\in M$. If we replace the para-complex bundle $E$ with $E_C$, then one can decompose $E_C$ into its $\pm j$-eigenbundles, and much of the complex theory can be replicated, including the Dolbeault sequence. In order to extend the operator $\overline{\partial}$ defined in equation \eqref{Dolbeault} to one compatible with our understanding of the complexification, we consider the following proposition from \cite{Krahe}:
\begin{prop}\label{real-para-complex}
	For a para-complex manifold $M$, let $\Lambda^{(p,q)}M_C$ denote the forms of type $(p,q)$ with respect to the para-complexified tangent bundle $TM_C$. Then there is an ($\mathbb{R}$-linear) isomorphism
	\begin{align*}
	\varphi: \Lambda^{(p,q)}M\times \Lambda^{(q,p)}M &\rightarrow \Lambda^{(p,q)}M_C
	\\
	(\eta ,\eta') &\mapsto \frac{1+j}{2}\eta + \frac{1-j}{2}\eta'
	\end{align*}
	such that the following diagram commutes:
	\\
	\begin{center}
		\begin{tikzcd}
			\Omega^{(p,q)}M\times \Omega^{(q,p)}M \arrow[r, "\varphi"]\arrow[d, "\partial_{-}\times \partial_{+}"] & \Omega^{(p,q)}M_C \arrow[d, "\bar{\partial}"]\\	\Omega^{(p,q+1)}M\times \Omega^{(q+1,p)}M \arrow[r, "\varphi"]	& \Omega^{(p,q+1)}M_C
		\end{tikzcd}.
	\end{center}
\end{prop}
\begin{proof}
	Suppose that $J$ is a para-complex structure on $M$. Then there are coordinates $z_{+}^{\alpha}$ on $T^{+}M$ and $z^{\alpha}_{-}$ on $T^{-}M$, which allow us to define para-holomorphic coordinates $z^{\alpha}$ in the usual way, $z^{\alpha} := \frac{1+j}{2}z^{\alpha}_{+} + \frac{1-j}{2}z^{\alpha}_{-}$. From the fact that $\left( \frac{1\pm j}{2} \right)^2 = \frac{1\pm j}{2}$ and $\frac{1+j}{2}\frac{1-j}{2} =0$, we can see that $\frac{1\pm j }{2}dz^{\alpha} = \frac{1\pm j}{2}dz_{\pm}^{\alpha}$ and $\frac{1\pm j}{2}d\bar{z}^{\alpha} = \frac{1\pm j}{2}dz^{\alpha}_{\mp}$. This means that if $\eta\in \Omega^{(p,q)}M$, then $\frac{1+j}{2}\eta\in \Omega^{(p,q)}M_C$, and if $\eta' \in \Omega^{(q,p)}M$, then $\frac{1-j}{2}\eta'\in \Omega^{(p,q)}M_C$. Therefore, the image of $\varphi$ is what we claim it to be. It is also easy to see that $\varphi$ is $\mathbb{R}$-linear as
	\begin{align*}
	\varphi\left( (\eta,\eta') + \alpha(\omega,\omega') \right) &= \frac{1+j}{2}(\eta + \alpha\omega) + \frac{1-j}{2}(\eta' + \alpha \omega')
	\\
	&= \frac{1+j}{2}\eta + \frac{1-j}{2}\eta' + \alpha \left( \frac{1+j}{2}\omega + \frac{1-j}{2}\omega' \right) 
	\\
	&= \varphi (\eta ,\eta') + \alpha \varphi(\omega,\omega'),
	\end{align*}
	for any $\alpha \in \mathbb{R}$. Furthermore, $\ker \varphi = \{0\}$ as $\phi(\eta, \eta')=0$ implies that $\frac{1+j}{2}\eta = - \frac{1-j}{2}\eta'$, which by multiplying through by $\frac{1+j}{2}$ (or $\frac{1-j}{2}$) gives $\frac{1+j}{2}\eta =0$ (or $\frac{1-j}{2}\eta' =0$). Since $\eta$ and $\eta'$ are real-valued forms, we can conclude that they are the zero form on each of their respective spaces, thus $\varphi$ is an isomorphism. For reference, the inverse of $\varphi$ is given by $(\eta, \eta') = (Re(\omega)+Im(\omega), Re(\omega)-Im(\omega))$. Finally, we can show that the diagram commutes. Recall that $\frac{1\pm j}{2}\frac{\partial}{\partial z^{\alpha}_{\mp}} = \frac{1\pm j}{2}\frac{\partial}{\partial \bar{z}^{\alpha}}$, implying
	\begin{align*}
	\varphi(\partial_{-}\eta , \partial_{+}\eta') &= \frac{1+j}{2}\partial_{-}\eta + \frac{1-j}{2}\partial_{+}\eta'
	\\
	&= \frac{1+j}{2}\bar{\partial}\eta + \frac{1-j}{2}\bar{\partial}\eta'
	\\
	&= \bar{\partial}\left( \frac{1+j}{2}\eta + \frac{1-j}{2}\eta' \right)
	\\
	&= \bar{\partial}\varphi(\eta,\eta').
	\end{align*}
\end{proof}
\indent The interesting part of Proposition \ref{real-para-complex} is that it appears as though the operator $\partial_{-}\oplus \partial_{+}$ plays the same role as $\overline{\partial}$ does in the complex setting. In the complex setting, the operator $\overline{\partial}$ sets up the Dolbeault sequence, which is the cohomological sequence of spaces of differential forms, and maps type $(p,q)$-forms to type $(p,q+1)$-forms. For para-complex bundles, the stand-in for this operator is $\partial_{-}\oplus \partial_{+}$ acting on $(p,q)\oplus(q,p)$-forms, and so we have an analogous Dolbealt sequence in the para-complex setting:
\begin{center}
    \begin{tikzcd}
    \cdots \arrow[r," \partial_{-}\oplus \partial_+"]  & \Omega^{(p,q)}M\oplus \Omega^{(q,p)}M\arrow[r," \partial_{-}\oplus \partial_+"] & \Omega^{(p,q+1)}M\oplus \Omega^{(q+1,p)}M\arrow[r," \partial_{-}\oplus \partial_+"]&\cdots
    \end{tikzcd}.
\end{center}
The choice to swap $(p,q)$ for $(q,p)$ comes directly from Proposition \ref{real-para-complex} and ensures that, as in the complex case, $\Omega^{(p,0)}M_C \cong \Omega^{(p,0)}M\oplus \Omega^{(0,p)}M$ is a para-holomorphic vector bundle (Where in the complex setting, $\Omega^{(0,p)}(M)$ is holomorphic \cite{Moroianu}). We can extend both $\partial_{+}$ and $\partial_{-}$ to $\Omega^{(p,q)}E$ in the same way as in the complex case:
\begin{align}
    \partial_{\pm} (\omega_1,\cdots , \omega_{2n}) = (\partial_{\pm}\omega_1,\cdots, \partial_{\pm}\omega_{2n}),
\end{align}
where $\omega_i \in \Omega^{(p,q)}(M)$. Additionally, we can consider the bundle
\begin{align}
    \Omega^{(p,q)}(E):=\Omega^{(p,q)}(M)\otimes E_+ \oplus \Omega^{(q,p)}(M)\otimes E_{-}.
\end{align}
In the case where $E$ is para-holomorphic, we can extend the operator $\partial_- \oplus \partial_+$ to the local sections $\omega=(\omega_1^{+},\cdots , \omega_n^{+}) \oplus (\omega_1^{-},\cdots , \omega_n^{-})\in \Omega^{(p,q)}(E)$, where $\omega_i^{+}$ are local $(p,q)$-forms and the $\omega_{i}^{-}$ are local $(q,p)$-forms, in a way that does not depend on the chosen trivialization, as the $E_{\pm}$ trivialization is invariant under $\partial_{\mp}$. Both $\partial_{+}$ and $\partial_{-}$ satisfy the Leibniz rule, and therefore so does $\partial_{-}\oplus \partial_{+}$. Using this convention, we retain the sequence
\begin{center}
    \begin{tikzcd}
    \cdots \arrow[r,"\partial_-\oplus\partial_+"] & \Omega^{(p,q)}(E)\arrow[r,"\partial_-\oplus\partial_+"]& \Omega^{(p,q+1)}(E)\arrow[r,"\partial_-\oplus\partial_+"]&\cdots,
    \end{tikzcd}
\end{center}
with the special property that $(\partial_{-}\oplus \partial_+)^2 = 0$. 
\begin{Def}
An operator $\partial_{-}\oplus\partial_{+}: \Omega^{(p,q)}(E)\rightarrow \Omega^{(p,q+1)}(E)$ for all $0\leq p,q\leq n$ on a para-complex vector bundle $E$ that satisfies the Leibniz rule is called a pseudo-paraholomorphic structure. If, moreover $(\partial_{-}\oplus \partial_{+})^2 =0 $, then $\partial_- \oplus \partial_+$ is called a para-holomorphic structure. 
\end{Def}
Just as in the holomorphic case, we have the following theorem.
\begin{Thm}
A para-complex vector bundle $(E,J)$ is para-holomorphic if and only if it has a para-holomorphic structure.  
\end{Thm}
\subsubsection{Para-Hermitian Manifolds}
In contrast to the theory of complex manifolds, where the appropriate objects of study in the Hermitian context are Riemannian metrics, para-complex geometry is concerned with pseudo-Riemannian metrics of split signature. To begin:
\begin{Def}
    A pseudo-Riemannian metric $g$ is a smooth, symmetric, non-degenerate section of $E^{*}\otimes E^{*}$. Given a para-complex vector bundle $(E,J)$, $g$ is said to be compatible with $J$ if $g(JX,Y) = -g(X,JY)$ for any sections $X,Y\in \Gamma(E)$. In this case, we refer to $g$ as a para-Hermitian metric. We define the induced fundamental form $F\in \Lambda^2E^*$ by $F(X,Y) = g(X,JY)$.
\end{Def}
We now quote some useful results from Bejan \cite{Bejan}.
\begin{prop}\label{Maximally Isotropic}
	Let $(J,g)$ be a para-Hermitian structure on $E$. Then:
	\begin{enumerate}
		\item $J$ is a para-complex structure on $E$ and $\textrm{rank}(E) = 2n$.
		\item $g$ is a pesudo-Riemannian structure of signature $(n,n)$.
		\item The eigenbundles $E^{+}$ and $E^{-}$ are maximally isotropic with respect to $g$. 
	\end{enumerate}
\end{prop}
\begin{proof}
	Recall that $g$ is non-degenerate, so there exists $s_1\in \Gamma(E)$ such that $g(s_1,s_1)\neq 0$. We may also assume that $g(s_1,s_1)>0$, with the negative case being similar. By the definition of compatibility and the fact that $g$ is symmetric, we have that $s_1$ is orthogonal to $Js_1$ with respect to $g$. If $\textrm{rank}(E)>2$, then we let $s_2\in \Gamma(E)$ be orthogonal to both $s_1$ and $Js_1$ with respect to $g$. It follows that $Js_2$ is also orthogonal to both $s_1$ and $Js_1$, as well as $s_2$, and so by continuing this process, we get a local orthogonal basis $\mathcal{B} = \{s_1,Js_1,\cdots, s_n,Js_{n}\}$, and so $\textrm{rank}(E) = 2n$.
	\\
	
	\indent The signature of $g$ is $(n,n)$ as $g(s_i,s_i)>0$, and so $g(Js_i,Js_i)= - g(s_i,s_i)<0$. Let us denote $u_i = s_i+Js_i$ and $v_i = s_i - Js_i$. We can see that, as before, $\{u_i\}$ and $\{v_i\}$ are local bases for the eigenbundles $E^{+}$ and $E^{-}$ respectively. Finally, we simply have to show that $g$ is identically zero when restricted to either bundle. The maximality condition is satisfied as $\textrm{rank}(E^{\pm}) = n$, and $E = E^{+}\oplus E^{-}$. To this end, we see that for the basis vectors $\{u_i\}$, we have
	\begin{align*}
	g(s_i+Jsi, s_j+Js_j) &= g(s_i,s_j) + g(Js_i,Js_j) + g(Js_i,s_j) + g(s_i,Js_j)
	\\
	&= g(s_i,s_j) - g(s_i,s_j) + g(s_i,Js_j)- g(s_i,Js_j)
	\\
	&=0.
	\end{align*} 
	\indent And a similar calculation shows that $g$ also vanishes on $E^{-}$.
\end{proof}
In keeping with the previous section, the addition of a compatible pseudo-Riemannian metric induces a further reduction in the structure sheaf of the para-Hermitian bundle, consisting of the collection of transition functions over some choice of open cover of $M$. We have the well known result due to Bejan \cite{Bejan}:
\begin{Thm}
	A vector bundle $E$ over $M$ admits a para-Hermitian structure if and only if the structure group acting on its frame bundle can be reduced to the group of matrices of the form:
	\begin{align*}
	\left\{g_{\alpha\beta}(p)  = \begin{pmatrix}
	S & 0 \\
	0 & (S^t)^{-1}
	\end{pmatrix} \, \, \, \, \, : \, \, \, \, \, S\in GL_k(\mathbb{R}), \, \, p\in M\right\}.
	\end{align*}
\end{Thm}
\begin{proof}
	Suppose that $E$ admits a para-Hermitian structure $(J,g)$. The first thing to note is that we require a compatible pseudo-Riemannian metric, whereas in the complex case, the complex structure is enough to provide a reduction of $GL_{2n}(\mathbb{C})\hookrightarrow GL_n(\mathbb{R})$. In the complex case, it is possible to construct local linearly independent frames just using the complex structure, but the same method does not work in the para-complex case unless a compatible metric is introduced. This is a consequence of the classification of almost vector cross-product structures. With this in mind, we proceed as follows. 
	\\
	
	\indent Let $(E,\pi, M)$ be a vector bundle of rank $2n$, and let $J$ be an endomorphism of $E$ such that $J^2=Id_{E}$. Let $\varphi_{\alpha}: \pi^{-1}(U_{\alpha})\rightarrow U_\alpha\times \mathbb{R}^{2n}$, and $\varphi_\beta : \pi^{-1}(U_{\beta})\rightarrow U_{\beta}\times \mathbb{R}^{2n}$ be two local trivializations such that $U_{\alpha}\cap U_{\beta}\neq \varnothing$. Consider the action of $J$ on $\mathbb{R}^{2n}$, given by $J_{\alpha} = \varphi_{\alpha}\circ J\circ \varphi_{\alpha}^{-1}$ and $J_{\beta}= \varphi_{\beta} \circ J \circ \varphi_{\beta}^{-1}$. Furthermore, we let $\varphi_{\beta}\circ \varphi_{\alpha}^{-1}(p,v) = (p,\tau_{\alpha\beta}(p)(v))$ for $(p,v)\in U_{\alpha}\cap U_{\beta}\times \mathbb{R}^{2n}$. Recall that $g$ is non-degenerate, so there exists $s_1\in \Gamma(E)$ such that $g(s_1,s_1)\neq 0$. We may also assume that $g(s_1,s_1)>0$, with the negative case being similar. By the definition of compatibility and the fact that $g$ is symmetric, we have that $s_1$ is orthogonal to $Js_1$ with respect to $g$. If $\textrm{rank}(E)>2$, then we let $s_2\in \Gamma(E)$ be orthogonal to both $s_1$ and $Js_1$ with respect to $g$. It follows that $Js_2$ is also orthogonal to both $s_1$ and $Js_1$, as well as $s_2$, and so by continuing this process, we get a local orthogonal frame $\mathcal{B}_{\alpha} = \{s_1,J_{\alpha}s_1,\cdots, s_n,J_{\alpha}s_{n}\}$ on the image of $\varphi_{\alpha}$. Following as in the complex case, we get a reduction into matrices of the form 
	\begin{align*}
	\tau_{\alpha\beta} &= \begin{pmatrix}
	a_{11} & b_{11} & \cdots & a_{n1} & b_{n1} \\
	b_{11} & a_{11} & \cdots & b_{n1} & a_{n1} \\
	\vdots & \vdots & \ddots & \vdots & \vdots \\
	a_{1n} & b_{1n} & \cdots & a_{nn} & b_{nn} \\
	b_{1n} & a_{1n} & \cdots & b_{nn} & a_{nn}
	\end{pmatrix}
	\end{align*}
	\indent In the para-complex setting, we may use the natural $\pm1$-eigenbasis, which is given by $\mathcal{B}_{\alpha}=\{\frac{1}{\sqrt{2}}(s^1+J_{\alpha}s^1),\cdots, \frac{1}{\sqrt{2}}(s^n+J_{\alpha}s^n),\frac{1}{\sqrt{2}}(s^1-J_{\alpha}s^1),\cdots, \frac{1}{\sqrt{2}}(s^n-J_{\alpha}s^n)\}$. In this basis, the transformation takes the form
	\begin{align}
	\tau_{\alpha\beta} &= \begin{pmatrix}
	a_{11}+b_{11} & a_{21}+b_{21} & \cdots & 0 & 0 \\
	a_{12}+b_{12} & a_{22}+b_{22} & \cdots & 0 & 0 \\
	\vdots & \vdots & \ddots & \vdots &\vdots \\
	0 & 0 & \cdot & a_{(n-1)(n-1)}-b_{(n-1)(n-1)} & a_{n(n-1)}-b_{n(n-1)} \\
	0 & 0 & \cdots & a_{(n-1)n}-b_{(n-1)n}&  a_{nn}-b_{nn}
	\end{pmatrix} = \begin{pmatrix}
	S & 0 \\
	0 & V
	\end{pmatrix}
	\end{align}
	\indent Now, we make an observation: $\tau_{\alpha\beta}$ maps the $\pm 1$-eigenspace of $J_{\alpha}$ into the $\pm 1$-eigenspace of $J_{\beta}$. With this in mind, we construct a basis $\mathcal{B}_{\beta}=\{\frac{1}{\sqrt{2}}(s^1+J_{\beta}s^1),\cdots, \frac{1}{\sqrt{2}}(s^n+J_{\beta}s^n),\frac{1}{\sqrt{2}}(s^1-J_{\beta}s^1),\cdots, \frac{1}{\sqrt{2}}(s^n-J_{\beta}s^n)\}$. Using this basis, we remark that
	\begin{align*}
	J_{\beta} =\begin{pmatrix}
	I_n & 0 \\
	0 & -I_n
	\end{pmatrix}, \, \, \, \, \, \, \, \, \, \, \, g_{\beta} = \begin{pmatrix}
	0 & I_n \\
	I_n & 0
	\end{pmatrix}.
	\end{align*}
	Knowing that $\tau_{\alpha\beta}$ vanishes on complimentary eigenspaces allows us to keep the form that we arrived at in equation (7.2.1). We also note that by definition $g_{\alpha} = \tau_{\alpha\beta}^{t} \circ g_{\beta}\circ \tau_{\beta\alpha}$, and so
	\begin{align*}
	\begin{pmatrix}
	0 & I_n \\
	I_n & 0 
	\end{pmatrix} &= \begin{pmatrix}
	S^t & 0 \\
	0 & V^t
	\end{pmatrix} \begin{pmatrix}
	0 & I_n \\
	I_n & 0 
	\end{pmatrix} \begin{pmatrix}
	S & 0\\
	0 & V
	\end{pmatrix}
	\\
	\begin{pmatrix}
	0 & I_n \\
	I_n & 0 
	\end{pmatrix} &= \begin{pmatrix}
	0 & S^t V\\
	V^t S & 0 
	\end{pmatrix}.
	\end{align*} 
	And thus we have $S^tV = 0$, or $V = (S^t)^{-1}$. Therefore, the structure group has a reduction into matrices of the form 
	\begin{align*}
	\tau_{\alpha\beta}(p)  = \begin{pmatrix}
	S(p) & 0 \\
	0 & (S^t)^{-1}(p)
	\end{pmatrix}, \, \, \, \, \, \, \, \, \, \, S(p)\in GL_k(\mathbb{R}),  \, \, p\in M.
	\end{align*}
	\indent Conversely, suppose that $\tau_{\alpha \beta}$ has the desired form. Simply picking
	\begin{align*}
	g\vert_{U_{\alpha}}= g_{\alpha} = \begin{pmatrix}
	0 & I_n\\
	I_n & 0 
	\end{pmatrix}
	\end{align*}
	and then extending $g_{\alpha}$ to all of $E$ as  
	\begin{align*}
	g_{\beta} = \tau_{\alpha\beta}^t \circ g_{\alpha} \circ \tau_{\alpha\beta}  = \begin{pmatrix}
	0 & I_n\\
	I_n & 0 
	\end{pmatrix}
	\end{align*}
	gives a well defined metric. As well, we get two sub-bundles of $E$ as the invariant spaces of $\tau_{\alpha \beta}$. Simply defining $J$ to be the identity on one of them and minus the identity on the other gives a para-complex structure, and one can see that on the overlap,
	\begin{align*}
	J_{\beta} &= \tau_{\alpha\beta}^{-1}\circ J_{\alpha}\circ \tau_{\alpha\beta}
	\\
	&=\begin{pmatrix}
	S^{-1} & 0 \\
	0 & S^{t}
	\end{pmatrix}\begin{pmatrix}
	I_n & 0 \\
	0 & -I_n
	\end{pmatrix} \begin{pmatrix}
	S & 0 \\
	0 & (S^t)^{-1}
	\end{pmatrix}
	\\
	&= \begin{pmatrix}
	I_n & 0 \\
	0 & -I_n
	\end{pmatrix}.
	\end{align*}
	Thus, from the reduction of the structure group, we obtain a para-Hermitian structure. 
\end{proof}
In this case, if we assume that $E$ is also para-holomorphic, then it is quite easy to see that the transition functions have to be constant. This is because $\partial_{-}(S)_{ij}=0$, and so $S$ can only be a function of the $z^{+}$ variables. But then we could say that the transition functions $(S^T)^{-1}$ can only be functions of the $z_{-}$ variables. These matrices are inverse transposes of each other, and so one can see that they must be constant. It is quite easy to show that when $g$ is compatible with the almost para-complex structure, $F\in \Omega^{(1,1)}M$. If the para-complex structure is integrable and $F$ is closed, then we know by the local $\partial_+ \partial_+$-lemma that locally $F=\partial_+ \partial_-u$ for some real function $u$, called the \textit{local para-K\"{a}hler potential}. 
\begin{Def}
Let $(M,J,g,F)$ be an almost para-Hermitian manifold. We say that this structure is para-K\"{a}hler and refer to $g$ as the para-K\"{a}hler metric if and only if $J$ is integrable and $dF=0$. In summary:
\begin{align}
    g\textrm{ is para-K\"{a}hler}\Leftrightarrow \begin{cases}
    N^J&=0\\
    dF&=0
    \end{cases}.
\end{align}
\end{Def}
The utility of this definition is that since the associated $2$-form is closed, the para-K\"{a}hler metric $g$ can be written locally in term of the local para-K\"{a}hler potential, giving
\begin{align*}
    g_{i\overline{j}} &= \frac{\partial^2 u}{\partial z^i_+ \partial z^j_-} .
\end{align*}
Many interesting examples of para-complex, para-Hermitian, and para-K\"{a}hler manifolds exist. For an overview of some historical examples as well as a survey of the development of the field of para-complex geometry, see \cite{Cruceanu}, in particular section $3$. We will now present some illustrative and interesting examples.  
\begin{Example}\label{Connection giving para-Hermitian}
The natural almost para-K\"{a}hler structure on $T^{*}M$:
\end{Example}
This example is due Bejan in \cite{Bejan2}, and will be informative to have in mind when we talk about para-complex connections in section 4. Let $M$ be a smooth manifold and let $\nabla$ be a torsion free connection on $M$. The connection $\nabla$ induces a decomposition $T_\xi(T^{*}M) = H_\xi(T^{*}M)\oplus V_\xi(T^{*}M)$, where $V_\xi(T^{*}M)$ is the vertical bundle with respect to the projection $\pi : T^{*}M\rightarrow M$, and $H_{\xi}(T^{*}M)$ is the horizontal bundle induced by the connection. In local coordinates $\{x^i,p^i\}$ for $T^{*}M$, we have that
\begin{align*}
    V_{(x,p)}(T^{*}M) &= \left\{ v^i \frac{\partial}{\partial p^i} \vert v^i\in C^{\infty}(T^{*}M) \right\} ,
    \\
    H_{(x,p)}(T^{*}M) &= \left\{ v^i \frac{\partial}{\partial x^i} + w^i \frac{\partial }{\partial p^i} \, \,  \vert \, \,v^i,w^i\in C^{\infty}(T^{*}M),\textrm{ and }  w^i + \Gamma^{i}_{jk}(x)v^jp^{k} = 0\right\},
\end{align*}
where $\Gamma_{jk}^{i}$ are the Christoffel symbols of the connection $\nabla$. Under this decomposition, we can define an almost para-complex structure $J$ with $+1$-eigenbundle equal to the horizontal bundle, and $-1$-eigenbundle equal to the vertical bundle. The natural para-Hermitian metric in the coordinates $\{x^i,p^i\}$ is then given by
\begin{align*}
    g &= \begin{pmatrix}
    -2p^k\Gamma^{i}_{jk}(x) & \delta^i_j\\
    \delta_j^i & 0 
    \end{pmatrix}.
\end{align*}
It is easy to see that both $H(T^{*}M)$ and $V(T^{*}M)$ are maximally isotropic with respect to this metric. Further, Bejan gives us the following Theorem:
\begin{Thm}
With the above notation, $(M,g,J)$ is an almost para-K\"{a}hler structure on the total space of the cotangent bundle $T^{*}M$, whose fundamental 2-form $F$ satisfies $F = d\theta$ (where $\theta$ denotes the Liouville form), and thus coincides with the canonical symplectic structure on $T^{*}M$. Moreover if $\nabla$ has vanishing curvature, then the structure $(M,g,J)$ is para-K\"{a}hler.
\end{Thm}
\begin{Example}
The Pseudosphere $S^{(3,3)}$.
\end{Example}
\indent In complex geometry, there is the famous question of whether the sphere $S^6$ admits an integrable complex structure. There are many known almost complex structures, but none so far discovered are integrable. In this vein of exploration, Libermann in \cite{Libermann} constructs an almost para-Hermitian structure on the pseudosphere $S^{(3,3)}$ that is not integrable, where $S^{(3,3)}$ is the submanifold of $\mathbb{R}^7$ determined by the polynomial
\begin{align*}
    -x_1^2 -x_2^2 -x_3^2 + x_4^2 +x_5^2 + x_6^2 + x_7^2= 1.
\end{align*}
In \cite{Smolentsev}, Smolentsev reconstructs this almost para-Hermitian structure by considering the quotient $G_2^{*}/SL(3;\mathbb{R})$, where $G_2^{*}$ is the non-compact real form of the exceptional Lie group $G_2$, and shows that its Nijenhuis tensor does not vanish. Interestingly, in section 5.1 of \cite{Smolentsev}, Smolentsev finds an integrable para-complex structure on $S^{(3,3)}$ by showing that $S^{(3,3)}$ is diffeomorphic to the cylinder $D^3\times S^3$, and so the question of whether integrable para-complex structures exist on $S^{(3,3)}$ is closed.
\pagebreak
\section{ \normalsize Lie Algebroids, Courant Algebroids and Lie-Bialgebroids}
We now move towards a discussion of algebroids. Roughly, algebroids allow us to combine the generality of para-Hermitian vector bundles with the integrability results that only appear in the context of the tangent bundle. By endowing the vector bundle with a bracket that has certain properties, we can ask questions about integrability of the eigenbundles that we would be unable to do otherwise. The two main objects of study here are Lie algebroids and Courant algebroids. A Lie algebroid is the most straightforward generalization of the tangent bundle, and a Courant algebroid is a generalization of the bundle $TM\oplus T^{*}M$, as we discussed in the introduction. 
\\

\indent The objects that we intend to retain are the non-degenerate symmetric bracket $\langle \cdot ,\cdot \rangle$ on the vector bundle $E\rightarrow M$ and the almost para-complex structure $J$. Recall that the eigenbundles of $J$ are maximally isotropic with respect to $\langle \cdot ,\cdot \rangle$. The fundamental take-away from the concept of a para-Hermitian manifold is that when the subbundles $T^{\pm}M$ are integrable, they form subalgebroids of the Lie algebroid $TM$. Therefore, the natural object to consider is the Lie bialgebroid $(E_{+},E_{-})$ consisting of dual pairs of Lie algebroids. As we will see, Lie bialgebroids are related to Courant algebroids in precisely the way that makes the Courant algebroid viable for a generalization of para-Hermitian geometry, which we then discuss in the next section. 
\subsection{ \normalsize Lie Algebroids}
\indent In the context of vector bundles, we lack a concept of integrability for the eigenbundles of a para-complex structure. Further, another weakness is that there is no accompanying derivative operator to decompose that acts on functions to produce sections of the dual bundle (as opposed to sections of $\Omega^{\cdot}(M)\otimes E$). It's clear that the two objects that are missing are an anchor map $\rho : E\rightarrow TM$, which will allow us to pull back the exterior derivative on $M$, and a bracket on the space of sections of $E$. For this reason, we introduce the notion of Lie algebroid using the formalism in \cite{Roytenberg}.
\begin{Def}
A Lie algebroid is a vector bundle $A\rightarrow M$ together with an anti-symmetric bracket $[\cdot,\cdot]_A$ on the space of sections $\Gamma(A)$ and a bundle map $a:A\rightarrow TM$ called the anchor. The bracket and the anchor satisfy the following conditions:
\begin{enumerate}
    \item For all $\sigma ,\tau \in \Gamma(A)$, $a([\sigma,\tau]_{A}) = [a(\sigma),a(\tau)]_{TM}$.
    \item For all $\sigma,\tau\in \Gamma(A)$ and $f\in C^{\infty}(M)$, $[\sigma,  f \tau]_{A} = f[\sigma,\tau]_{A} + (a(\sigma)f)\tau $.
\end{enumerate}
\end{Def}
\indent One can use this bracket and anchor to produce a derivation of degree $1$, $d_A$, on the space of sections $\Gamma(\extp^{*}A^{*})$ by using an identical form of the Cartan formula. For any $\omega \in \Gamma(\extp^{p}A^{*})$, we can define $d_A\omega \in \Gamma(\extp^{p+1} A^{*})$ by 
\begin{align}\label{Cartan Formula}
    d_A\omega(\sigma_0,\cdots, \sigma_p) &= \sum_{i=0}^{p} (-1)^{i}(a(\sigma_i))\omega(\sigma_1,\cdots ,\hat{\sigma}_i , \cdots , \sigma_p ) 
    \\
    \nonumber & \, \, \, \, \, + \sum_{0\leq i < j \leq p}(-1)^{i+j}\omega([\sigma_i,\sigma_j]_{A},\sigma_0,\cdots, \hat{\sigma}_i,\cdots , \hat{\sigma}_j,\cdots , \sigma_p)
\end{align}
Note that $d_A^2=0$. Moreover, the Lie bracket on $A$ is recoverable from $d_A$ since for any $\omega \in \Gamma(\extp^1A^{*})$,
\begin{align*}
    \omega([\sigma , \tau]_A) = a(\sigma)\omega(\tau) - a(\tau)\omega(\sigma) - d_A\omega (\sigma, \tau).
\end{align*}
All of the structural identities of the Lie bracket are induced by the fact that $d^2_A = 0$ and can be discovered by considering $d^2_A g(\sigma , f \tau)$ for arbitrary $f,g\in C^{\infty}(M)$. Further, it is possible to extend the Lie bracket to multi-vector fields (sections of $\extp^{*}A$) via the following rules. For $\sigma \in \Gamma(\extp^{p}A)$, we let the degree of $\sigma$ be $|\sigma| = p $.  
\begin{enumerate}
    \item $[\sigma,\tau]_{A} = -(-1)^{(|\sigma|-1)(|\tau|-1)}[\tau,\sigma]$.
    \item For $\sigma \in \Gamma(A)$ and $f\in C^{\infty}(M)$, $[\sigma,f] = a(\sigma)f$.
    \item $[\sigma,\cdot]_A$ is a derivation of degree $|\sigma|-1$, meaning 
    \begin{align*}
    [\sigma , \tau\wedge  \xi]_A = [\sigma , \tau]_A\wedge \xi +(-1)^{(|\sigma|-1)(|\tau|)} \tau \wedge [\sigma,\xi]_A, \, \, \, \, \, \, \, \, \, \, |[\sigma,\tau]| = |\sigma| + |\tau|-1.
    \end{align*}
\end{enumerate}
\indent The derivation $d_{A}$ and interior derivative $\iota$, defined by contracting sections of $\extp A^*$ with sections of $A$, can be used to define the Lie derivative operator $\mathcal{L}^{A}_X = [d_A , \iota_X]$. All of the usual commutation relations satisfied by these operators continue to hold in this setting.
\begin{Example}
The tangent bundle, $TM$, to a smooth manifold $M$.
\end{Example}
As we remarked in the introduction to this section, the simplest Lie algebroid is the tangent bundle $TM$ with the Lie bracket as the bracket on the space of sections and the anchor given by the identity map. 
\begin{Example}
The action Lie algebroid $\mathfrak{g}\times M \rightarrow M$.
\end{Example}
\indent A Lie algebra action of $\mathfrak{g}$ on $M$ is a map $\sigma :\mathfrak{g}\rightarrow \mathfrak{X}(M)$ such that $\sigma ([X,Y]_{\mathfrak{g}}) = [\sigma(X),\sigma(Y)]_{TM}$ for $X,Y\in \mathfrak{g}$. If one considers the trivial bundle $\mathfrak{g}\times M$, we have an anchor given by $\rho (X,p) = \sigma(X)\vert_p $ and can extend the bracket on $\mathfrak{g}$ to $\mathfrak{g}\times M$ by derivations over smooth functions on $M$ to obtain a Lie algebroid called the \textit{action Lie algebroid} of $\mathfrak{g}$ on $M$. Given an element of the dual Lie algebra $\xi \in \mathfrak{g}^{*}$, one can compute the derivative of $\xi$ with respect to the induced derivation $d_\mathfrak{g}$ as 
\begin{align*}
d_{\mathfrak{g}}\xi (X,Y) &= \sigma(X)\xi(Y) - \sigma(Y)\xi(X)- \xi ([X,Y]_{\mathfrak{g}}),
\end{align*}
which when $X$ and $Y$ are constant sections, meaning elements of $\mathfrak{g}$, reduces to $d_\mathfrak{g}\xi (X,Y) = -\xi([X,Y]_{\mathfrak{g}})$, which is the standard Poisson structure on $\mathfrak{g}^{*}$ \cite{Symplectic Geometry}.
\subsection{ \normalsize Courant Algebroids}
\indent Now, the Lie algebroid structure may be sufficient to define a para-complex structure by simply requiring the eigenbundles to be closed under the bracket structure, but this definition does not extend easily to the para-Hermitian setting. The more natural objects in this setting are the Courant algebroid and Lie bialgebroid, which we also draw from \cite{Roytenberg}.
\begin{Def}
Consider the vector bundle $E\rightarrow M$ equipped with a symmetric bilinear form $\langle \cdot ,\cdot \rangle$, a skew-symmetric bracket $[\cdot,\cdot]_E : \extp^2E \rightarrow E$, and an anchor map $\rho : E\rightarrow TM$. We define $d_E:C^{\infty}(M)\rightarrow \Gamma(E^{*})$ by $d_E= \rho^* d$, meaning for a section $e\in \Gamma(E)$ and a smooth function $f\in C^{\infty}(M)$, $d_E f( e) = \rho(e)f$, and the map $T$ given by $T(e_1,e_2,e_3) = \frac{1}{6}\langle [e_1,e_2],e_3\rangle + c.p.$. Then $E$ is a Courant algebroid if the following conditions are satisfied. For all $e,e_1,e_2,e_3,h_1,h_2\in \Gamma(E)$ and $f\in C^{\infty}(M)$,
\begin{enumerate}
\item  $J(e_1,e_2,e_3) = \beta^{-1} d_ET (e_1,e_2,e_3)$,
\item $\rho([e_1,e_2]) = [\rho(e_1),\rho(e_2)]$,
\item $[e_1,fe_2] = f[e_1,e_2] + (\rho(e_1)f)e_2 - \frac{1}{2}\langle e_1,e_2 \rangle \beta^{-1}d_Ef$,
\item $\rho \circ \beta^{-1}\circ d_E = 0 $,
\item$\rho(e)\langle h_1,h_2\rangle = \langle [e,h_1],h_2\rangle + \langle h_1 , [e,h_2]\rangle + \frac{1}{2}d_E\langle e,h_1\rangle (h_2) + \frac{1}{2}d_E\langle e,h_2\rangle (h_1)$,
\end{enumerate}
where $J(e_1,e_2,e_3) = [[e_1,e_2],e_3] + [[e_2,e_3],e_1] + [[e_3,e_1],e_2]$ and $\beta: E\rightarrow E^{*}$ is the non-degenerate bundle isomorphism induced by $\langle\cdot,\cdot\rangle$.
\end{Def}
A special kind of Courant algebroid that will be useful for consideration later are exact Courant algebroids. An exact Courant algebroid is a Courant algebroid that admits a fiber-wise decomposition $E_x = T_xM\oplus T_x^{*}M$, $x\in M$. This is equivalent to the following sequence being exact:
\begin{center}
    \begin{tikzcd}
    0 \arrow[r,""]& T^{*}M\arrow[r,"\rho^{*}"]&E\arrow[r,"\rho"]& TM\arrow[r,""]&0,
    \end{tikzcd}
\end{center}
where in this case, $\rho^{*}$ is the dual map to $\rho$ under the pairing $\langle \cdot ,\cdot\rangle$, meaning $\langle \rho^{*}\xi , e \rangle = \xi (\rho(e))$. Exactness allows us to say that $\ker(\rho) = Im(\rho^{*})$ and $\ker(\rho^{*}) = 0 $. 
\\

\indent Given an exact Courant algebroid $E$, one can deform the Courant bracket on $E$ by any closed $3$-form in the following way.
\begin{Def}
Let $E\rightarrow M$ be an exact Courant algebroid and let $\eta \in \Omega^{3}(M)$. We define the $\eta$-twist of the Courant bracket $[\cdot,\cdot]$ by
\begin{align}
    [e_1,e_2]_\eta &= [e_1,e_2] + \iota_{e_1}\iota_{e_2}\rho^{*}\eta.
\end{align}
Then, $[\cdot,\cdot]_\eta$ is a Courant bracket precisely when $\eta$ is closed \cite{Roytenberg}.
\end{Def}
\indent In the definition of the Courant algebroid, we have introduced a derivative operator $d_E$ that will play a special role. Note that $d_E$ can be extended to all of $\extp^{*}E^{*}$ via the Cartan formula, as in equation \eqref{Cartan Formula}. There remains the question of what morphisms between Courant algebroids should look like. At the bare minimum, we should ask that a map $\Psi:E\rightarrow F$ between Courant algebroids preserves the anchor and the symmetric bilinear form (meaning that $\Psi$ should be an isometry). The way that $\Psi$ should interact with the bracket is informed by the following lemma. 
\begin{Lem}
Let $E,F$ be Courant algebroids over $M$ and $\Psi:E\rightarrow F$ be a vector bundle map that preserves the metric and anchor ($\rho_F \circ \Psi = \rho_E$). Then $\Psi\circ d_E = d_F$ if and only if $\Psi$ preserves the Courant brackets up to an element of $\rho^{*}_E(\Omega^2(M))\otimes \rho_F^*(\Omega^1(M))$.
\end{Lem}
\begin{proof}
\indent All that we have to show is that if $\Psi \circ d_E = d_F$, then $\Psi([e_1,e_2]_E)- [\Psi(e_1),\Psi(e_2)]_F$ is tensoral in both arguments. Since the brackets are anti-symmetric, we need only show this for one entry. For $e_1,e_2\in \Gamma(E)$, $f\in C^{\infty}(M)$, we have
\begin{align*}
\Psi([e_1,fe_2]_E) - [\Psi(e_1),\Psi(fe_2)]_F &= \Psi\left( f[e_1,e_2]_E + \rho_E(e_1)f e_2 -\frac{1}{2} \langle e_1,e_2\rangle d_E f  \right)
\\
& \, \, \, \, \, \, \, \, \, \,  -\left( f[\Psi(e_1),\Psi(e_2)]_F +\rho_F(\Psi(e_1))f e_2 - \frac{1}{2}\langle \Psi(e_1),\Psi(e_2)\rangle d_F f \right)
\\
&= f\left( \Psi([e_1,e_2]_E) - [\Psi(e_1),\Psi(e_2)]_F \right)
\\
& \, \, \, \, \, \, \, \, \, \, \,  + \frac{1}{2}\langle \Psi(e_1),\Psi(e_2)\rangle (\Psi(d_Ef) - d_Ff).
\end{align*}
And so we can see that $\phi= \Psi([e_1,e_2]_E)- [\Psi(e_1),\Psi(e_2)]_F$ is tensoral if and only if $\Psi\circ d_E = d_F$. By the fact that $\Psi$ preserves the anchor, we find that $\rho_{F}(\phi(e_1,e_2)) = 0$, and so by exactness $\phi(e_1,e_2)= \rho_F^{*}(\omega_{e_1,e_2})$, for $\omega \in \Omega^{1}(M)$. Finally, since $E$ and $F$ are exact, we can dualize the expression $\rho_E = \rho_F \circ \Psi$ to obtain $\rho_F^{*} = \Psi \circ \rho_E^{*}$, and hence $\phi(\rho_E^{*}(\xi), \rho^{*}_{E}(\zeta)) = 0$ for all $\eta,\zeta\in \Omega^{1}(M)$, meaning $\phi \in \rho_E^{*}(\Omega^{2}(M))\otimes \rho_F^{*}(\Omega^1(M))$, as desired. 
\end{proof}
\noindent Based on this understanding of the way that isometric vector bundle morphisms that preserve the anchor interact with the bracket, we introduce the following definition.
\begin{Def}
Let $E,F\rightarrow M$ be two exact Courant algebroids and let $\Psi:E\rightarrow F$ be an isometry that preserves the anchors of $E$ and $F$. Then we call $\Psi$ an \textit{almost Courant algebroid morphism} if $\Psi\circ d_E=d_F$. We call $\Psi$ \textit{Courant algebroid morphism} if it preserves the bracket.
\end{Def}
If $\Psi:E\rightarrow E$ is an almost Courant algebroid automorphism, then $\Psi([\cdot,\cdot])$ is the Courant bracket on $E$ up to a $\phi$-twist by a closed form $\phi \in \Omega^3(M)$. We mention these maps in lieu of Example \ref{Para-Hermitian Lie Groups}, wherein it makes more sense to represent the double tangent bundle to a Lie group $\mathbb{T}G$ as $G\times (\mathfrak{g}\oplus \overline{\mathfrak{g}})$, but a Courant algebroid morphism is needed to make this identification in a way that allows us to extract useful information about the natural Courant algebroid on $\mathbb{T}G$. The goal now is to examine the situation where $E$ admits a decomposition into dual subbundles for which the restriction of the Courant bracket corresponds to a Lie bracket on the subbundles. 
\subsection{ \normalsize Lie Bialgebroids}
\indent There are specific subbundles of Courant algebroids that are of interest in the classical theory, namely, Dirac subbundles. A \textit{Dirac subbundle} is a maximally isotropic subbundle $L$ whose sections are closed under the bracket on $E$. In this case, we see that the function $T$ from the definition of a Courant algebroid has the property that $T\vert_L=0$ by the fact that $L$ is Lagrangian and closed under the bracket, and so the Jacobiator $J$ is identically zero by condition $1$ for Courant algebroids, meaning the bracket $[\cdot,\cdot]_E$ is a Lie bracket on $L$. Further, for any $l\in\Gamma(L)$, $[l,\cdot]_E$ acts by derivations on sections of $L$, and so $L\rightarrow M$ is a Lie algebroid. We now consider the special case where $E= A\oplus A^{*}$ for two Dirac subbundles $A$ and $A^{*}$ of $E$. Our choice of notation here comes from the fact that $A$ and $A^{*}$ must be maximally isotropic and hence disjoint, and so we can identify the second subspace as the dual to $A$ under the inner-product. In this case, we have the following definition.
\begin{Def}
Suppose that $A$ and $A^{*}$ are both Lie algebroids over $M$. We say that the pair $(A,A^{*})$ is a Lie bialgebroid if $d_A$ is a derivation of the Schouten bracket on $A^{*}$. I.e., $d_{A}[\xi,\eta]_{A^{*}} = [d_{A}\xi, \eta]_{A^{*}} + [\xi , d_A\eta]_{A^{*}}$ for $\xi ,\eta \in \Gamma(\extp A^{*})$.  
\end{Def}
It is possible to cast the theory of Lie bialgebroids in terms of Courant algebroids using the following theorems from \cite{Manin Triples}.
\begin{Thm}\label{bialgebroid -  Courant Algebroid}
Let $(A,A^{*})$ be a Lie bialgebroid with corresponding anchors $a$ and $a_{*}$. Then $(E= A\oplus A^{*}, \rho ,\langle \cdot , \cdot\rangle_{+} , [\cdot,\cdot])$ is a Courant algebroid, where $\rho = a + a_{*}$, $\langle X + \xi , Y+\eta \rangle_{\pm} = \langle \xi , Y \rangle \pm \langle \eta , X\rangle $, and for $e_1 = X + \xi$ and $e_2 = Y+\eta$, 
\begin{align}\label{Induced Bracket}
    [e_1,e_2] &= ([X,Y]_A + \mathcal{L}^{A^{*}}_\xi Y - \mathcal{L}_\eta^{A^{*}}X - \frac{1}{2}d_{A^{*}} \langle e_1,e_2\rangle_-)
    \\
    & \nonumber \, \, \, \, \,  + ([\xi,\eta]_{A^{*}} + \mathcal{L}^{A}_{X}\eta - \mathcal{L}^{A}_{Y}\xi + \frac{1}{2} d_{A}\langle e_1,e_2\rangle_{-})
\end{align}
\end{Thm}
\indent We will return to this choice of $\langle \cdot ,\cdot \rangle_{-}$ later as it will have significance in the para-Hermitian setting. It is easy to see that, in this case, we have a splitting of the derivative operator $d_E = d_A + d_A^{*}$. This construction is dual in the sense that if one begins with a Courant algebroid, we have the following:
\begin{Thm}\label{Courant-biagebroid}
Let $(E,\rho, \langle\cdot,\cdot\rangle , [\cdot,\cdot])$ be a Courant algebroid, and suppose that $L_1$ and $L_2$ are two maximally isotropic subbundles of $E$ that are both closed under the Courant bracket. Then $E= L_1\oplus L_2$ and $(L_1,L_2)$ is a Lie bialgebroid, where $L_2$ is considered the dual to $L_1$ under $\langle\cdot,\cdot \rangle$. 
\end{Thm}
\begin{Def}
A Courant algebroid $E$ that admits a direct sum decomposition $E = A \oplus A^{*}$ into isotropic subbundles is called a \textit{proto-bialgebroid}. If one of $A$ or $A^{*}$ is closed under the bracket, then $E$ is called a \textit{quasi-Lie bialgebroid}.
\end{Def}
The obstruction to any Courant algebroid $E$ being a proto-bialgebroid is simply the existence of a dual pair of Lagrangian subbundles. Given a proto-bialgebroid $E = A\oplus A_{*}$, we can ask what the obstruction to closure with respect to the bracket on $E$ for each of the subbundles is. Given sections $\sigma_1,\sigma_2\in \Gamma(A)$ and $\tau_1,\tau_2 \in \Gamma(A^{*})$, we define the sections $\phi\in\Gamma(\extp^{3}(A^{*}))$ and $\psi\in \Gamma(\extp^3(A))$ by 
\begin{align}\label{almost Nijenhuis}
    \iota_{\sigma_1}\iota_{\sigma_2}\phi &= \pi_{A^{*}} [\sigma_1,\sigma_2]_{E},
    \\
   \nonumber \iota_{\tau_1}\iota_{\tau_2} \psi &= \pi_{A} [\tau_1,\tau_2]_{E}.
\end{align}
We can be sure that $\phi$ is a $3$-form on $A$, as the bracket is skew-symmetric, and for any $f\in C^{\infty}(M)$,
\begin{align*}
[\sigma_1,f\sigma_2]_{E} &= f[\sigma_1,\sigma_2]_E + \rho(\sigma_1)f \sigma_2 - \frac{1}{2}\langle \sigma_1,\sigma_2\rangle d_Ef
\\
&= \left( f\pi_{A}([\sigma_1,\sigma_2]_E) + \rho(\sigma_1)f\sigma_2 \right) \oplus f\pi_{A^{*}}([\sigma_1,\sigma_2]_{E}),
\end{align*}
where at the end we have split the sum into its $A$ and $A^{*}$ part. We can easily see that $\pi_{A^{*}}[\cdot,\cdot]$ is tensoral from the final expression. We also obtain a proto-Lie bracket $[\cdot,\cdot]_A$ on $A$ which fails to satisfy the Jacobi identity defined as $[\cdot,\cdot]_{A} := \pi_{A}[\cdot ,\cdot ]\vert_{A}$. We can do the same thing for $A^{*}$, and obtain $\psi\in \Gamma(\extp^{3}(A))$ that measures the failure of $A^{*}$ to be closed and a proto-Lie bracket $[\cdot,\cdot]_{A^{*}} = \pi_{A^{*}}[\cdot,\cdot]\vert_{A^{*}}$. From the discussion on quasi-bialgebroids in \cite{Roytenberg}, in particular equation 3.18,  we arrive at the following theorem.
\begin{Thm}\label{Obstruction-Lie-Bialgebroid}
Let $E=A\oplus A^{*}$ be a proto-bialgebroid. The obstructions to the closure of the subbbundles $A$ and $A^{*}$ under the Courant bracket are the $3$-forms $\phi\in \Gamma(\extp^{3}(A^*))$ and $\psi\in \Gamma(\extp^3(A))$. Further $d_{A}\phi= d_{A^{*}}\psi = 0$ for any proto-bialgebroid, where $d_A$ and $d_A^{*}$ are defined in equation \eqref{Cartan Formula} using $\pi_A[\cdot,\cdot]_E$ and $\pi_A^{*}[\cdot,\cdot]_E$ as the stand-in for the Lie bracket respectively.
\end{Thm}
This final condition that $d_A\phi = d_{A^{*}}\psi = 0$ will have more significance in the para-Hermitian setting. We are now in a position to define our main object of study. 
\pagebreak
\section{ \normalsize Para-Hermitian Algebroids}
\indent We now have all of the ingredients necessary to define a para-Hermitian algebroid. The goal will be to cast all of the previous facts about Courant and Lie algebroids in terms of the para-Hermitian structure. To begin, we have
\begin{Def}
Let $(E,\rho , \langle\cdot,\cdot\rangle , [\cdot,\cdot])$ be a Courant algebroid over $M$. Then $E$ is called an almost para-Hermitian algebroid if it is equipped with a para-complex structure $J$ that is compatible with the inner product (i.e., $\langle J \cdot, J\cdot \rangle = - \langle \cdot ,\cdot \rangle$). We call $(E,J)$ half integrable if one of the eigenbundles of $J$ is closed under the Courant bracket, and integrable if both are closed under the Courant bracket. In the case where both eigenbundles are integrable, we refer to $(E,J)$ as a para-Hermitian algebroid.
\end{Def}
\indent Courant algebroids are anchored to the tangent bundle of the underlying manifold, and so we are lead to consider para-holomorphic algebroids, which occur when a para-Hermitian algebroid, $(E,J_E)$, is anchored to a para-Hermitian manifold, $(M,J_{TM})$, and the anchor map, $\rho :E\rightarrow TM$, is para-holomorphic. Another program in this chapter will be to generalize the construction in Example \ref{Connection giving para-Hermitian} by considering connections $A:TM\rightarrow E$. As we will see, connections are themselves para-complex structures, and just as in the case of para-holomorphic algebroids, it makes sense to consider the case where they commute with another para-complex structure on $E$. We complete this generalization in Example \ref{Normal} in section 5, which one can think of as containing the companion examples to this section. 
\subsection{ \normalsize Integrability and the Para-K\"{a}her Condition}
\indent In order to measure the integrability of the eigenbundles, one can construct the Nijenhuis tensor as in equation \eqref{Nijenhuis Tensor}. Immediately, we see from Proposition \ref{Maximally Isotropic} that if the eigenbundles are integrable, then they are maximally isotropic, and so by Theorem \ref{Courant-biagebroid}, $E= E_+ \oplus E_{-}$ and $(E_+,E_-)$ is a Lie bialgebroid. By the discussion in the previous section, both $E_+$ and $E_-$ form Lie algebroids under the restricted brackets. This allows us to define the derivative operators
\begin{align}
    d_\pm : C^{\infty}(M)\rightarrow \Gamma(E^{*}_\pm)
\end{align}
with the property that $d_E = d_+ + d_-$. These operators then give us a type decomposition of forms. If we identify $E_{\pm}^{*}\cong E_{\mp}$ under $\langle \cdot , \cdot \rangle$ and consider them as subbundles of $E$, then
\begin{align}
    \extp^{(1,0)}E^{*} &:= \{\omega \in E^{*} \vert \omega (\sigma^{*}) = 0  \, \, \, \forall \sigma^{*}\in \Gamma(E_{-})\},
    \\
    \nonumber \extp^{(0,1)}E^{*} &:= \{\omega \in E^{*} \vert \omega (\sigma) = 0  \, \, \, \forall \sigma\in \Gamma(E_{+})\}.
\end{align}
The two derivations $d_+$ and $d_-$ are only traditionally defined on forms of type $\extp^{(p,0)}E^{*}$ and $\extp^{(0,q)}E^{*}$, respectively, but we see that under this splitting,
\begin{align}
    d_E : \extp^{(p,q)}E^{*}\rightarrow \extp^{(p+1,q)}E^{*} \oplus \extp^{(p,q+1)}E^{*},
\end{align}
and so we can in general define $d_{+}$ to be the projection onto the $(p+1,q)$ subbundle and $d_-$ to be the projection onto the $(p,q+1)$ subbundle, as in the classical para-Hermitian case. Note also that this decomposition is only possible when both eigenbundles are integrable, and so we have an analogous result to Theorem \ref{Classical Decomposition of Forms}. It is also important to note that the derivative operators act by derivations on the induced Schouten brackets associated to the dual Lie algebroid, so for any $e_{1}^{\pm}, e_2^{\pm} \in \bigwedge E_{\pm}^*$,
\begin{align}\label{Derivational}
    d_{\pm}[e_1^{\mp},e_2^{\mp}]_{\mp} = [d_{\pm}e_{1}^{\mp} , e_2^{\mp}] + [e_{1}^{\mp} , d_{\pm}e_{2}^{\mp}].
\end{align}
\indent Each of these operators can be expressed in local form. Consider the local basis $\{e_1^{+},\cdots,e_n^{+},e_1^{-},\cdots,e_n^{-}\}$, and suppose that
\begin{align}
    a_{+}(e_i^{+}) &= A_{i}^{j}(x) \frac{\partial}{\partial x^j}, & a_{-}(e_i^{-}) = A_{\overline{i}}^{j}(x) \frac{\partial}{\partial x^j},
\end{align}
where $a_{\pm} = \rho\circ \pi_{E_{\pm}}$, and
\begin{align}
    [e_i^{+},e_j^{+}] &= C_{ij}^k(x) e_k^{+} & [e_i^{-},e_j^{+}] &= C_{\overline{i}j}^{k}(x)e_k^{+} + C_{\overline{i}j}^{\overline{k}}(x)e_k^{-}  & [e_i^{-},e_j^{-}] &= C_{\overline{i}\overline{j}}^{\overline{k}}(x)e_k^{-}
\end{align}
From these two facts, one can readily compute that for $\sigma^{\pm}\in \Gamma(E_{\pm})$ locally expressed as $\sum_i\sigma_i^\pm e_i^\pm$,
\begin{align}
\nonumber d_+(\sigma^{+}) &= -\sum_{i,j}\left(A_{j}^{k}(x) \frac{\partial \sigma_i^{+}}{\partial x^k} + C_{\overline{i}j}^{k}(x) \sigma_k^{+}\right) e_i^+ \wedge e_j^- ,
\\
d_{-}(\sigma^{+}) &= -\sum_{i,j}\left( A_{\overline{j}}^{k}(x) \frac{\partial \sigma_i^+}{\partial x^k} - A_{\overline{i}}^k (x)\frac{\partial \sigma_j^+}{\partial x^k} + C_{\overline{i}\overline{j}}^{\overline{k}}(x)\sigma_k^{+}\right)e_i^{+}\wedge e_j^{+},
\\
\nonumber d_{+}(\sigma^{-}) &= -\sum_{i,j}\left( A_{j}^{k}(x)\frac{\partial \sigma_i^{-}}{\partial x^k} - A_{i}^{k}(x) \frac{\partial \sigma_j^{-}}{\partial x^k} + C_{ij}^{k}(x) \sigma^{-}_k \right) e_i^{-}\wedge e_j^{-},
\\
\nonumber d_- (\sigma^{-}) &= -\sum_{i,j}\left( A_{\overline{j}}^k(x) \frac{\partial \sigma_i^{-}}{\partial x^k} + C_{i\overline{j}}^{\overline{k}}(x)\sigma_k^{-} \right)e_i^{-}\wedge e^{+}_j.
\end{align}
\indent Just as in the case of a para-holomorphic vector bundle, we have $d_{\pm}(\sigma^i e_i^{\pm}) = \partial_{\pm}(\sigma^i)\otimes e_i^{\pm} + \sigma^i d_{\pm}(e_i^{\pm})$. It would be useful at this point to define what we mean by para-holomorphic forms. We consider a section $\sigma\oplus \tau \in \extp^{(p,q)}E^{*}\oplus \extp^{(q,p)}E^{*}$ to be \textit{para-holomorphic} if $(d_{-}\oplus d_{+})(\sigma \oplus \tau) = 0$. In particular, there are two ways for a section $\sigma \in \Gamma(E)$ to be para-holomorphic. By identifying $E\cong E^{*}$ under the non-degenerate pairing on $E$, one can decompose $\sigma = \sigma^{+}\oplus \sigma^{-} \in \extp^{(0,1)}E^{*}\oplus \extp^{(1,0)}E^{*}$, or $\sigma = \sigma^{-}\oplus \sigma^{+} \in \extp^{(1,0)}E^{*}\oplus\extp^{(0,1)}E^{*}$. In this way, $\sigma$ can be para-holomorphic as a $(0,1)\oplus(1,0)$-form or as a $(1,0)\oplus (0,1)$-form.
\\

\indent Returning to Theorem \ref{Obstruction-Lie-Bialgebroid}, the notion of pseudo-para-holomorphic structures on almost para-Hermitian algebroids gives an interesting interpretation of the integrability condition for the eigenbundles. Consider the bundle $E_{+}$. The Nijenhuis tensor when restricted to $E_{+}$ is given by 
\begin{align}\label{Nijenhuis Eigenbundle}
    N^J(e_i^+, e_j^+) = \frac{1}{2}\left( [e_i^{+},e_j^{+}]_E - J[e_i^+,e_j^+] \right),
\end{align}
which when compared to equation \eqref{almost Nijenhuis} tells us that $\phi(e_i^{+},e_j^{+},e_k^{+}) =\langle e_k^{+}, N^J(e_{i}^{+},e_j^{+})\rangle + c.p.$. This can be done for $E_{-}$ as well, to obtain $\psi(e_i^{-},e_j^{-},e_k^{-}) = \langle e_k^{-},N^J(e_i^{-},e_j^{-})\rangle + c.p.$. Stretching this analogy further, we have $d_{+}\phi=0$ and $d_{-}\psi=0$, and so the obstruction to the integrability of the brackets is a pseudo-para-holomorphic $(3,0)\oplus (0,3)$-form corresponding to the restriction of the Nijenhuis bracket onto each eigenbundle. This was first realized by Svoboda in \cite{Svoboda}. 
\\

\indent In light of this, and Theorem \ref{bialgebroid -  Courant Algebroid}, we can investigate how the bracket on $E$ is constructed from the given data of the Lie bialgebroid $(A,A^{*})$. In particular, we take notice of $\langle\cdot,\cdot \rangle_{-}$. In terms of the metric on $E$, which we denote $\langle \cdot,\cdot \rangle$, and the para-complex structure $J$ with $+1$-eigenbundle $A$ and $-1$-eigenbundle $A^{*}$, we define the fundamental 2-form on $E$ by
\begin{align}
    \omega = \langle  \cdot ,J\cdot  \rangle.
\end{align}
\noindent Note that when the eigenbundles are integrable, this form is of type $(1,1)$. Recall that in the case where the eigenbundles are integrable, $E_+$ and $E_{-}$ are isotropic, meaning $\omega(e_i^{\pm},e_j^{\pm})=0$. The ``para-K\"{a}hler" condition, usually given by $d_E \omega= 0$, can be equivalently stated as $d_\pm\omega =0$. We can use this expression to obtain a PDE constraining the coefficients of $\omega$. To this end, we have the following.
\begin{Lem}
    The coefficients of the fundamental $2$-form $\omega$ associated to the para-Hermitian Courant algebroid $E= E_{+}\oplus E_{-}$ satisfy
    \begin{align}
     A_{\overline{i}}^l(x) \frac{\partial \omega_{\overline{j}k}}{\partial x^l} - A_{\overline{j}}^{l}(x)\frac{\partial \omega_{\overline{i}k}}{\partial x^l}- C_{ij}^{l}(x) \omega_{l\overline{k}} + C_{i\overline{k}}^{\overline{l}}(x) \omega_{\overline{l}j} - C_{j\overline{k}}^{\overline{l}}(x)\omega_{\overline{l}i}&=0,
    \\
    A_i^{l}(x) \frac{\partial \omega_{j\overline{k}}}{\partial x^l} - A_{j}^{l}(x) \frac{\partial \omega _{i\overline{k}}}{\partial x^l} -C_{ij}^{l}(x)\omega_{l\overline{k}} + C_{i\overline{k}}^{\overline{l}}(x)\omega_{\overline{l}j} - C_{j\overline{k}}^{\overline{l}}(x)\omega_{\overline{l}i} &= 0.
    \end{align}
\end{Lem}
\begin{proof}
This follows easily from equation \eqref{Cartan Formula}, along with the substitutions of vectors from each desired eigenspace. The first equation corresponds to $d_+\omega=0$ and the second corresponds to $d_-\omega=0$.
\end{proof}
\subsection{ \normalsize Exact Almost Para-Hermitian Algebroids}
Recall that for an almost para-Hermitian algebroid $E = E_{+}\oplus E_{-}$, the projections onto each eigenbundle induce a splitting in the anchor $\rho :E\rightarrow TM$ as the sum of the two operators $a_{\pm}:E_{\pm}\rightarrow TM$.  This splitting of the anchor map induces the following sequences:
\begin{center}
    \begin{tikzcd}
    0\arrow[r,""] & T^*M \arrow[r,"a_{-}^{*}"] & E_{+} \arrow[r,"a_{+}"] & TM\arrow[r,""]&0 ,
    \end{tikzcd}
    \begin{tikzcd}
    0\arrow[r,""] & T^*M \arrow[r,"a_{+}^{*}"] & E_{-} \arrow[r,"a_{-}"] & TM\arrow[r,""]&0 .
    \end{tikzcd}
\end{center}
\begin{Lem}
    Let $E=E_{+}\oplus E_{-}$ be a Lie bialgebroid, with anchor $\rho$. If $\rho \circ \rho^{*} = 0 $, then $a_{\pm}\circ a_{\mp}^{*} + a_{\mp}\circ a_{\pm}^{*}=0$
\end{Lem}
\begin{proof}
To begin, suppose that $E=E_{+}\oplus E_{-}$ is a Lie bialgebroid. We find
\begin{align*}
    \rho \circ \rho^{*} (X) &=  (a_{+} + a_{-}) (a_{+}^{*}(X) + a_{-}^{*}(X)) 
    \\
    &= a_{+}(a_{+}^{*}(X)) + a_{+}(a_{-}^{*}(X)) + a_{-}(a_{+}^{*}(X))+ a_{-}(a_{-}^{*}(X))
    \\
    &= a_{+}(a_{-}^{*}(X)) + a_{-}(a_{+}^{*}(X))
\end{align*}
where we have used the fact that $a_{\pm}\circ a_\pm^{*}=0$, owing to the fact that $E_{+}\oplus E_{-}$ is a Lie bialgebroid. The statement then follows by the fact that $\rho \circ \rho^{*}=0$.
\end{proof}
This gives us the following useful proposition.
\begin{prop}\label{Exact Hermitian}
Let $(E,J)$ be an almost para-Hermitian algebroid. If $E= E_{+}\oplus E_{-}$ is exact, then $\ker(a_{\pm}) = Im(a_{\mp}^{*})$. 
\end{prop}
\begin{proof}
Suppose that $E$ is exact, so that $\ker(\rho) =Im(\rho^{*})$. Explicitly, $\rho(e) = 0$ if and only if $e = \rho^{*}(\xi)$ for some $\xi \in\Gamma( T^{*}M)$. Suppose that $a_{\pm}(e^{\pm})=0$ for some $e^{\pm}\in \Gamma(E_{\pm})$. Since $e^{\pm}$ is a section of the eigenbundle, $a_{\pm}(e^{\pm}) = \rho(e^{\pm})$, and so $e^{\pm} = \rho^{*}(\xi)$ for some $\xi \in T^{*}M$. However, since $\rho = a_{+} + a_{-}$, we can decompose $e^{\pm} = a_{-}^{*}(\xi) + a_{+}^{*}(\xi)$. This tells us that $a_{\pm}^{*}(\xi) = 0$, as the image is in the opposite eigenbundle to $e^{\pm}$, and so we must conclude that $e^{\pm} = a_{\mp}^{*}(\xi)$. In conclusion, $a_{\pm}(e^{\pm}) =  0$ if and only if $e^{\pm} = a_{\mp}^{*}(\xi)$ for some $\xi \in \Gamma(T^{*}M)$, and so $\ker(a_{\pm}) = Im(a_{\mp}^{*})$.
\end{proof}
\subsection{ \normalsize Para-Holomorphic Algebroids}
\indent Recall that all para-holomorphic, para-Hermitian vector bundles in the traditional sense have constant transition functions. 
\begin{Def}
Let $E\rightarrow M$ be a para-Hermitian algebroid over a para-Hermitian manifold. We say that $E$ is  a para-holomorphic algebroid if $\rho\circ J_E = J_{TM}\circ \rho $.
\end{Def}
\indent One immediate consequence of the condition that the para-complex structures commute with $\rho$ is that the splitting of $\rho$ into $a_{+}+a_{-}$ has the property that the image of the maps land in their respective eigenbundles of the para-Hermitian structure on $M$. By this we mean that $a_{\pm} : E_{\pm}\rightarrow T^{\pm}M$. If we consider that $d_{\pm}f(e) = df(a_{\pm}(e))$, we are led to conclude that the image of $e$ under $a_{\pm}$ is in $T^{\pm}M$. In fact, the condition $d_{\pm}f(e) = \partial_{\pm}f(a_{\pm}(e))$ is equivalent to saying that $a_{\pm}$ is anchored in $T^{\pm}M$. Since $TM$ is para-Hermitian, we see that $T^{(1,0)}M \cong (T^{+}M)^{*}$ and $T^{(0,1)}M\cong (T^{-}M)^{*}$. This sets up two sequences:
\begin{center}
    \begin{tikzcd}
    0\arrow[r,""] & T^{(0,1)}M \arrow[r,"a_{-}^{*}"] & E_{+} \arrow[r,"a_{+}"] & T^{+}M\arrow[r,""]&0 ,
    \end{tikzcd}
    \begin{tikzcd}
    0\arrow[r,""] & T^{(1,0)}M \arrow[r,"a_{+}^{*}"] & E_{-} \arrow[r,"a_{-}"] & T^{-}M\arrow[r,""]&0 .
    \end{tikzcd}
\end{center}
We have the following.
\begin{prop}
Suppose that $E = E_{+}\oplus E_{-}$ is a para-holomorphic algebroid. Then $E$ is exact if and only if the previous two sequences are exact.
\end{prop}
\begin{proof}
 If both of these sequences are exact, then we can simply note that taking the termwise direct sum gives us that $E$ is exact since each of the individual terms are vector bundles whose direct sums are the desired bundles in the sequence for $E$ and $\rho = a_+ + a_{-}$, so we focus on the other direction. Suppose that $E$ is exact. From Proposition \ref{Exact Hermitian}, we know that $\ker(a_{\pm}) = Im(a_{\mp}^{*})$, and so we simply need to check that $\ker(a_{\mp}^{*})=0$ and that the maps $a_{\pm}$ are surjective. Surjectivity follows from the fact that the images of $a_{\pm}$ are disjoint, and the map $\rho = a_{+}+a_{-}$ is surjective by exactness. Additionally, we know that $\ker(\rho^{*})=0$ by exactness, and so if $a_{\pm}^{*}(\xi) = 0$, then clearly $\rho^{*}(\xi)=0$ as well, and this completes the proof.
\end{proof}
\noindent We immediately have the following corollary.
\begin{Coro}
If $E$ is an exact para-holomorphic algebroid, then $\dim(E) = 4k$ for some $k\in \mathbb{N}$.  
\end{Coro}
Continuing, let $E$ be an exact almost para-Hermitian algebroid. We can relate the anchors $a_{\pm}$ of the eigenbundles $E_{\pm}$ to a natural bivector field on $M$. More concretely, there is a natural bivector on $M$ that is Poisson when $E$ is para-Hermitian, and whose vanishing is a requirement for $E\rightarrow M$ to be para-holomorphic. Consider the bivector $\pi \in \mathfrak{X}^2(M)$ given by
\begin{align}\label{Poisson on M}
    \iota_{\mu} \pi  = \rho (pr_{E_{-}}(\rho^{*}(\mu))) = -\rho (pr_{E_{+}}(\rho^{*}(\mu))),
\end{align}
where the second equality comes from the fact that $E$ is exact. It was shown in \cite{Courant Algebroid Splitting} that the rank of the bundle map $\pi^{\#}: T^{*}M\rightarrow TM$ is given by $\textrm{rank}(\pi^{\#}) = \dim(\rho(E_{+})\cap \rho(E_{-}))$ and so this bivector is identically $0$ when the images under the anchor $\rho$ of $E_{+}$ and $E_{-}$ are disjoint, as is the case when $E\rightarrow M$ is para-holomorphic. Further, they deduce that the obstruction to $\pi$ being a Poisson bivector is 
\begin{align}
\frac{1}{2}[\pi,\pi]_s &= \rho(\phi) + \rho (\psi),
\end{align}
where $\phi$ and $\psi$ are defined as in equation \eqref{almost Nijenhuis}, we have identified $E^{*}_{\pm}\cong E_{\mp}$ and $[\cdot,\cdot]_s$ is the Schouten bracket derived from the Lie bracket on $TM$. This allows us to state the following.
\begin{Thm}
Let $E$ be an exact almost para-Hermitian algebroid over a para-Hermitian manifold $M$. If $E$ is para-Hermitian, the bivector field $\pi$, as defined in equation \eqref{Poisson on M}, is Poisson. If $E\rightarrow M$ is para-holomorphic, then $\pi = 0$. 
\end{Thm}
\indent We can now consider morphisms of para-Hermitian algebroids. In this case, the model is the morphism of para-Hermitian manifolds. A morphism of para-Hermitian manifolds is a smooth map $\Phi : (M,J_M)\rightarrow (N,J_N)$ such that $\Phi_*$ is an isometry, and $\Phi_{*}\circ J_M = J_N\circ \Psi_{*}$. Note that by the properties of pushforwards, $[\Phi_{*}(\cdot),\Phi_{*}(\cdot)]_N = \Phi_{*}[\cdot,\cdot]_M$. With this in mind, we can define a morphism of para-Hermitian algebroids.
\begin{Def}
Let $(E,J_E)$ and $(F,J_F)$ be para-Hermitian algebroids. A vector bundle map $\Psi:E\rightarrow F$ is a morphism of para-Hermitian algebroids if $\Psi$ is an isometry, para-holomorphic (meaning $\Psi \circ J_E = J_F\circ \Psi$) and preserves the Courant bracket.
\end{Def}
It is important to note here that the definition of a para-Hermitian algebroid morphism is less strict than a morphism of Courant algebroids because we do not require that these maps preserve the anchor. We will see in Example \ref{Para-Hermitian Lie Groups} that these maps can allow us to construct interesting examples in cases where Courant algebroid morphisms are too strict of a condition. We can introduce a notion of a para-holomorphic algebroid morphism as well.
\begin{Def}
Let $(E,J_F)$ and $(F,J_F)$ be para-holomorphic algebroids over $(M,g_1,J_1)$ and $(N,g_2,J_2)$ respectively. Suppose that $\Psi:E\rightarrow F$ is a vector bundle morphism covering the diffeomorphism $\psi : M\rightarrow N$. We say that $\Psi$ is a morphism of para-holomorphic algebroids if $\Psi$ is a para-Hermitian algebroid morphism, and the diffeomorphism $\psi : M\rightarrow N$ is a para-holomorphic isometry. 
\end{Def}
\subsection{ \normalsize Connections on Para-Hermitian Algebroids}
Bressler and Chervov give a notion of a connection on a Courant algebroid in \cite{Bressler} that will be of use to us.
\begin{Def}\label{Connection}
Let $E$ be a Courant algebroid. A connection on $E$ is a vector bundle morphism $A: TM\rightarrow E$ that satisfies:
\begin{enumerate}
    \item $\rho \circ A = Id_{TM}$.
    \item $\langle A(v_1), A(v_2) \rangle = 0$ for all $v_1,v_2\in TM$.
\end{enumerate}
\end{Def}
\indent In the context of exact Courant algebroids, connections are important in that they provide a decomposition $E = \rho^{*}(T^{*}M)\oplus A(TM)$. One can see that the dimensions work out, as $\rho^{*}$ and $A$ are injective (by exactness and $\rho \circ A = Id_{TM}$, respectively). Finally, we see that $\ker(\rho) = Im(\rho^{*})$, but $\rho \circ A = Id_{TM}$, and so we can conclude that these subbundles are disjoint. We can define the para-complex structure $K$ with eigenbundles $A(TM)$ and $\rho^{*}(T^{*}M)$. This para-complex structure is always half integrable precisely because $E$ is an exact Courant algebroid. We know that $\ker(\rho) = Im(\rho^{*})$, and additionally
\begin{align*}
    \rho([\rho^{*}(\xi),\rho^{*}(\eta)])= [\rho(\rho^{*}(\xi)),\rho (\rho^{*}(\eta))]=0,
\end{align*}
so $[\rho^{*}(\xi),\rho^{*}(\eta)] = \rho^{*}(\zeta)$. This is exactly the condition needed for $\rho^{*}(T^*M) $ to be closed under the bracket on $E$. We can now consider the related concept of curvature for these connections. 
\begin{Def}\label{Curvature}
Let $A:TM \rightarrow E$ be a connection. The curvature of $A$ is a map $R: TM \times TM \rightarrow E$ defined by 
\begin{align}
    R(v_1,v_2) = [A(v_1),A(v_2)]_{E} - A([v_1,v_2]).
\end{align}
In short, $R$ is the measure of the failure of the connection to preserve the bracket on $TM$ and $E$.
\end{Def}
One then has the following theorem that relates the integrability of the para-complex structure $K$ on an exact Courant algebroid $E$ to this concept of curvature. 
\begin{Thm}\label{Standard para-complex structure}
Let $E$ be an exact courant algebroid, and let $A:TM \rightarrow E$ be a connection on $E$. Then $E = A(TM)\oplus \rho^{*}(T^{*}M)$ has the structure of a half integrable para-complex algebroid, with $K(A(v)) = A(v)$ and $K(\rho^{*}(\xi)) = -\rho^{*}(\xi)$. We call $K$ the \textit{standard para-complex structure with respect to $A$} on $E$. The para-complex structure on $E$ is fully integrable if and only if the connection on $E$ is flat. \end{Thm}
\begin{proof}
It is clear that since $K$ is half integrable, one would only need to check that $A(TM)$ is closed under the bracket if and only if $A$ is flat. For this, we suppose that $v_1, v_2\in \Gamma(TM)$. Then $A(TM)$ is closed if and only if there exists $v_3\in \Gamma(TM)$ such that 
\begin{align*}
[A(v_1),A(v_2)] &= A(v_3).
\end{align*}
We note that $\rho$ preserves the bracket and that $\rho \circ A = Id_{TM}$, and so by applying $\rho$ to both sides, we find that $v_3 = [v_1,v_2]$. This is precisely the condition that $A$ is flat. As for the fact that $K$ is compatible with the metric, we find
\begin{align*}
    \langle K(A(v)+ \rho^{*}(\xi)) , K(A(w) + \rho^{*}(\eta))\rangle  &= -\langle A(v)+ \rho^{*}(\xi) , A(w) + \rho^{*}(\eta)\rangle,
    \\
    \langle A(v)-\rho^{*}(\xi), A(w)-\rho^{*}(\eta)\rangle  &= - \langle A(v)+\rho^{*}(\xi),A(w)\rangle ,
    \\
    \langle A(v),A(w)\rangle  + \langle \rho^{*}(\xi), \rho^{*}(\eta)\rangle &= -\langle A(v),A(w)\rangle  - \langle \rho^{*}(\xi), \rho^{*}(\eta)\rangle,
    \\
    \langle \rho^{*}(\xi),\rho^{*}(\eta)\rangle &=-\langle \rho^{*}(\xi),\rho^{*}(\eta)\rangle.
\end{align*}
Finally we see that $\langle \rho^{*}(\xi),\rho^{*}(\eta)\rangle =0$ for all $\xi,\eta \in \Gamma(T^{*}M)$, as $\Im(\rho^{*})$ isotropic by exactness.
\end{proof}
\begin{Example}
Consider the standard Courant algebroid $E = TM\oplus T^{*}M$ for some smooth manifold $M$, with $\rho : E\rightarrow TM$ given by projection on the first factor. Then $\rho^{*}$ is closed under this bracket in the trivial sense, as $[\xi, \eta ]= 0$. In this case, $Im(\rho^{*})$ is isotropic as well. One can check rather easily that all connections on $E$ are in the form $A(X) = X \oplus \tilde{\omega}(X)$ for some $\omega \in \Omega^2(M)$. This is the graph of a $2$-form in $E$, and so by the famous result that the Dirac subbundles of $T\oplus T^{*}$ are closed 2-forms, $(E,K)$ is a para-Hermitian algebroid (i.e., the image of the connection is integrable) if and only if $\omega$ is a closed 2-form. So, for $E = TM\oplus T^{*}M$, para-Hermitian algebroid structures with one eigenbundle equal to $Im(\rho^{*})$ are in one-to-one correspondence with closed $2$-forms.
\end{Example}
\indent Para-complex and para-Hermitian algebroids are more general than connections, but para-Hermitian structures with one eigenbundle equal to $\rho^{*}(T^{*}M)$ are in one-to-one correspondence with connections on $E$, as we justify in the following proposition.
\begin{prop}
Given an exact Courant algebroid, we can make a choice of subbundle $H$ such that $E = \rho^{*}(T^{*}M)\oplus H$. Then $\rho\vert_{H}:H\rightarrow TM$ is an isomorphism. If $H$ is also isotropic (meaning this is a half integrable para-Hermitian algebroid), then $H$ defines a connection on $E$ by $\rho^{-1}: TM \rightarrow H \hookrightarrow E$. 
\end{prop}
\subsection{ \normalsize Para-Complex Connections}
\indent We can now move on to the more general case of a para-complex structure that does not necessarily come from a connection or $Im(\rho^{*})$. Let $E\rightarrow M$ be an almost para-Hermitian algebroid over an almost para-Hermitian manifold $M$. We say that a connection $A$ on $E$ is para-complex if $J_E \circ A = A\circ J_{TM}$. If this is the case, we immediately arrive at the following.
\begin{Lem}
If $A$ is a para-complex connection, then $A$ maps the $\pm 1$-eigenspace of $J_{TM}$ into the $\pm 1$-eigenspace of $J_{E}$. 
\end{Lem}
\begin{proof}
Suppose that $A$ is a para-complex connection on $E\rightarrow M$. Then for $v\in T^{\pm}M$, 
\begin{align*}
J_{E}(A(v)) &= A(J_{TM}(v)) = A(\pm v) = \pm A(v).
\end{align*}
\end{proof}
One immediately finds that if $A$ preserves the bracket, there is a sense in which the structure of the underlying almost para-Hermitian manifold is determined by the structure of the almost para-Hermitian algebroid that covers it. We can formalize this in the following theorem.
\begin{Thm}
 Let $E$ be a para-Hermitian algebroid over an almost para-Hermitian manifold $M$.  Then $M$ is para-Hermitian if $E$ admits a flat para-complex connection.
\end{Thm}
\begin{proof}
To begin, suppose that $E$ is a para-Hermitian algebroid, so that both of its eigenbundles are closed under the Courant bracket. Recall that the image of $T^{\pm}M$ under $A$ is in $E_{\pm}$. From this, we can see that for $v_1,v_2\in \Gamma(T^{\pm}M)$, 
\begin{align*}
    A(J_{TM}[v_1,v_2]) &= J_{E}(A([v_1,v_2]))
    \\
    &= J_E([A(v_1),A(v_2)]_{E})
    \\
    &= \pm [A(v_1),A(v_2)]
    \\
    &= A(\pm [v_1,v_2]).
\end{align*}
We can now use the fact that $\rho\circ A = Id_{TM}$, so that $J_{TM}[v_1,v_2] = \pm [v_1,v_2]$, meaning $[v_1,v_2]\in \Gamma(T^{\pm}M)$. Therefore, the eigenbundles of $J_{TM}$ are closed under the Lie bracket, and $M$ is a para-Hermitian manifold. 
\end{proof}
If $A$ is para-complex we can decompose $A$ into two bundle maps $A = A_{+}+ A_{-}$, given by $A_{\pm} : T^{\pm}M \rightarrow E_{\pm}$. One immediately has the identity $a_{\pm}\circ A_{\pm} = Id_{T^{\pm}M}$. Further, we know that the image of $A$, and both $E_{\pm}$ are isotropic, and so given two $v,w\in \Gamma(TM)$, decomposed by $v= v_{+}+v_{-}$, $w= w_{+}+w_{-}$ in the usual way, we have
\begin{align*}
    \langle A_{+}(v_{+}), A_{-}(w_{-})\rangle &= -\langle A_{-}(v_{-}), A_{+}(w_{+})\rangle .
\end{align*}
In any case, the left and right hand sides of this equation are completely independent, and so we can conclude that the images of $A_{+}$ and $A_{-}$ are orthogonal under this metric. This means that $E_{+}$ contains a subbundle, $A_{+}(T^{+}M)$, that is orthogonal to another subbundle of $E_{-}$, namely $A_{-}(T^{-}M)$. From the fact that $a_{\pm}\circ A_{\pm} = Id_{T^{\pm}M}$, we can conclude that the maps $A_{\pm}$ are injective, and hence bijective onto their image. We can once again return to the concept of the exact para-holomorphic algebroid. Given a connection $A$ on $E$, one can construct two complimentary diagrams:
\begin{center}
    \begin{tikzcd}[row sep=large, column sep=small]
    0&\arrow[l,""] T^{-}M&\arrow[l,"a_-"]E_-  &\arrow[l,"a_+^{*}"] T^{(1,0)}M\arrow[l,""] &\arrow[l,""] 0 \arrow[r,""] &T^{(0,1)}M \arrow[r,"a_-^{*}"] & E_{+}\arrow[r,"a_{+}"] & T^{+}M\arrow[r,""] & 0
    \\
     & &  T^{-}M\arrow[u,"A_{-}"]\arrow[ul,"Id"] &  &   & &T^{+}M\arrow[u,"A_{+}"]\arrow[ur,"Id"] &  &
    \end{tikzcd}.
\end{center}
We know from exactness that $\ker(a_{\pm}) = Im(a_{\mp}^{*})$ and $a_{\pm} \circ A_{\pm} = Id_{TM}$, and so we can conclude that the images of $a_{\mp}^{*}$ and $A_{\pm}$ are disjoint. Further, we know that both maps $a_{\pm}^{*}$ and $A_{\pm}$ are injective. This implies that the eigenbundles admit a direct sum decomposition into two half-dimensional subbundles. We arrive at a series of interesting conclusions.
\begin{Thm}\label{Main Thm}
If $(E,J)$ is an exact para-Hermitian algebroid over a para-Hermitian manifold, and admits a flat para-complex connection $A$, then $E$ is a para-holomorphic algebroid. 
\end{Thm}
\noindent This follows directly from the previous discussion. This structure allows us to split each eigenbundle into the direct sum of two distinct subbundles coming from the para-holomorphic algebroid structure and the para-complex connection. We summarize this in the following theorem. 
\begin{Thm}
Suppose that $E$ is an exact para-Hermitian algebroid over a para-Hermitian manifold. If $E$ admits a para-complex connection $A: TM\rightarrow E$, then the eigenbundles of the para-complex structure on $E$ admit the following decomposition.
\begin{align*}
    E_{+} &= a_{-}^{*}(T^{(0,1)}M)\oplus A_{+}(T^{+}M),
    \\
    E_{-} &= a_{+}^{*}(T^{(1,0)}M)\oplus A_{-}(T^{-}M).
\end{align*}
\end{Thm}
Now, this decomposition could equally have been written 
\begin{align*}
E\cong (A_{+}(T^{+}M)\oplus A_{-}(T^{-}M))\oplus (a_{-}^{*}(T^{(0,1)}M)\oplus a_{+}^{*}(T^{(1,0)})).
\end{align*}
\noindent Note that in this case, we clearly have $JK=KJ$ where $K$ is the standard para-complex structure with respect to $A$ on $E$ as defined in Theorem \ref{Standard para-complex structure}, and so we can construct a third para-complex structure $L = JK$. In the case where $Im(\rho^{*})$ is isotropic, then $J$ and $K$ are both para-Hermitian, and so the resulting para-complex structure $L$ is \textit{split-para-complex}, as $\langle L \cdot , L\cdot \rangle = \langle \cdot , \cdot \rangle $. This structure of two commuting para-complex structures (and the induced third para-complex structure) is referred to as a \textit{split-para-complex structure}, and since two of the structures are para-Hermitian, we refer to this as a \textit{split-para-Hermitian structure}. As far as I am aware, this is the first naturally occurring instance of a split-para-Hermitian structure, though similar structures have been considered before. We can conclude with this theorem
\begin{Thm}
If $(E,J)$ is an exact para-holomorphic algebroid, then $E$ admits a flat para-complex connection if and only if $(E,J,K)$ is split-para-complex.   
\end{Thm}
\begin{proof}
We discussed already the forward direction here, so we focus on proving that a split-para-complex structure (with one para-complex structure being the standard one with one eigenbundle equal to $\rho^{*}(T^{*}M)$) is equivalent to defining a para-complex connection. To begin, we know that the choice of para-Hermitian structure $K$ is equivalent to choosing a decomposition $E = \rho^{*}(T^{*}M)\oplus H$, where $H$ is half dimensional and isotropic. We also remarked that $\rho\vert_{H}$ must be an isomorphism, and so $A:=(\rho\vert_{H})^{-1}: TM\rightarrow H\hookrightarrow E$ defines a connection on $E$. Our goal now is to show that $A$ is a para-complex connection in the case where $(E,J,K)$ is split-para-complex and $(E,J)$ is para-holomorphic. We focus on the subbundle $H$. Since $JK=KJ$ and $H$ is the $-1$-eigenbundle of $K$, we have $- J\vert_{H} = K\circ J\vert_{H}$, and so $J\vert_{H} : H\rightarrow H$. Further, since $\rho\vert_{H}$ is invertible and $(E,J)$ is para-holomorphic, we have $(\rho\vert_{H})^{-1}\circ J_{TM} = J\vert_{H}\circ (\rho\vert_{H})^{-1}$, and so $(\rho\vert_{H})^{-1}$ is a para-complex connection. In particular, the eigenbundle $H$ is integrable if and only if $A$ is a flat connection.
\end{proof}
Returning briefly to the concept of a para-Hermitian algebroid morphism, given two para-holomorphic algebroids $(E,J_E)$ and $(F,J_F)$ over the para-Hermitian manifold $M$, then a para-Hermitian algebroid morphism $\Psi$ does not necessarily map connections to connections. Given $A:TM\rightarrow E$, the map $\Psi\circ A : TM\rightarrow F$ does not necessarily have the property that $\rho_F \circ \Psi \circ A = Id_{TM}$, and so we cannot conclude that $A$ is a connection.  In fact, the condition for this to be true is precisely that $\Psi$ must preserve the anchor which is equivalent to $\Psi$ being a Courant algebroid morphism. The only obstruction to $\Psi \circ A$ being a connection is the one previously mentioned, and so $\Psi \circ A$ is still a bivector on $F$. Further, we have the following. 
\begin{prop}
Let $(E,J_E)$ and $(F,J_F)$ be exact para-Hermitian algebroids over a para-Hermitian manifold $M$, and let $\Psi : E\rightarrow F$ be an isomorphism of para-Hermitian algebroids. When a connection $A:TM\rightarrow E$ is flat, $\Psi \circ A$ is a Poisson bivector on $F$. 
\end{prop}
\begin{proof}
Since $\Psi$ is an isometry, $\langle \Psi\circ A,\Psi\circ A \rangle = \langle A , A \rangle = 0 $, so the image of $A$ is isotropic, meaning $A$ is the graph of a bivector on $F$. Further Since $A$ is flat, $[\Psi\circ A,\Psi \circ A ]_{F} = \Psi ([A,A]_E) = \Psi \circ A [\cdot,\cdot]_M$, so $\Psi \circ A$ preserves the Courant bracket and is therefore a Poisson bivector.
\end{proof}
\pagebreak
\section{ \normalsize Examples}
\subsection{\normalsize The natural structures on $\mathbb{T}M$}
\begin{Example}\label{Normal}
The natural para-holomorphic structure with compatible para-complex connection on a para-Hermitian manifold.
\end{Example}
\indent Let $(M,g,J)$ be an almost para-Hermitian manifold. The splitting of the tangent bundle into $TM=T^{+}M\oplus T^{-}M$ induces a type decomposition on the space of differential forms, and so we have $T^{*}M = T^{(1,0)}M\oplus T^{(0,1)}M$. We endow $\mathbb{T}M= TM\oplus T^{*}M$ with the standard Courant algebroid structure as described in the introduction and equation \eqref{Standard Algebroid}. Now, the splitting of the tangent and the cotangent bundles allows us to introduce an almost para-complex structure $\tilde{J}$ on $\mathbb{T}M$, with $+1$-eigenbundle $T^{+}M\oplus T^{(0,1)}M$ and $-1$-eigenbundle $T^{-}M\oplus T^{(1,0)}M$. We can note that for any $X^{+}\in \Gamma(T^{+}M)$, $X^{-}\in\Gamma( T^{-}M)$, $\xi^{(1,0)}\in \Gamma(T^{(1,0)}M)$ and $\xi^{(0,1)}\in \Gamma(T^{(0,1)}M)$, we have $\iota_{X^{+}}\xi^{(0,1)}= \iota_{X^{-}}\xi^{(1,0)}=0$, so this structure is para-Hermitian as the eigenbundles are isotropic. Further, we can easily see that the para-complex structures $J$ and $\tilde{J}$ commute with respect to the projection onto $TM$, so we have an almost para-holomorphic structure. Turning now to the obstruction to integrability, we can compute that 
\begin{align}\label{Standard eigenbundle contraction}
    [X^{+}\oplus \xi^{(0,1)}, Y^{+}\oplus \eta^{(0,1)}] &= [X^{+},Y^{+}]\oplus \left( \iota_{X^{+}}d\eta^{(0,1)} + \iota_{Y^{+}}d\xi^{(0,1)} \right),
    \\
    \nonumber [X^{-}\oplus \xi^{(1,0)}, Y^{-}\oplus \eta^{(1,0)}] &= [X^{-}, Y^{-}]\oplus \left( \iota_{X^{-}}d\eta^{(1,0)} + \iota_{Y^{-}}d\xi^{(1,0)} \right).
\end{align}
\indent Now, we did not assume that the eigenbundles $T^{\pm}M$ were integrable, and so not only is $[X^{\pm},Y^{\pm}]$ not guaranteed to be a section of $T^{\pm}M$, but we fail to satisfy the conditions in Theorem \ref{Classical Decomposition of Forms}, and so 
\begin{align*}
    d\Omega^{(1,0)}(M), d\Omega^{(0,1)}(M) \subset \Omega^{(2,0)}(M)\oplus \Omega^{(1,1)}(M) \oplus \Omega^{(0,2)}(M).
\end{align*}
One can check that the bivector field on $M$ defined by equation \eqref{Poisson on M} is identically zero, as $\rho  (\pi_{+}(\rho^{*}(\xi)))= \rho (0\oplus \xi^{(0,1)}) = 0$, which is to be expected because $\mathbb{T}M$ is a para-holomorphic algebroid. By Theorem \ref{Classical Decomposition of Forms}, we can see that if $T^{+}M$ is integrable, then $d\Omega^{(0,1)}(M) \subset \Omega^{(1,1)}(M) \oplus \Omega^{(0,2)}(M)$. Since in equation \eqref{Standard eigenbundle contraction} we are contracting $d\eta^{(0,1)}$ and $d\xi^{(0,1)}$ with elements of $\Gamma(T^{+}M)$, the part of this that remains is the $\Omega^{(1,1)}(M)$-part, and so upon contracting with such a section, we are left with an element of $\Omega^{(0,1)}(M)$, as desired. There was nothing special about choosing $T^{+}M$ here, so the integrability of the eiegenbundle $T^{\pm}M$ implies the integrability of the $\pm1$-eigenbundle of the para-Hermitian structure on $\mathbb{T}M$ and vice versa. All are equivalent to the Nijenhuis tensor vanishing, as discussed in the paragraph after equation \eqref{Nijenhuis Eigenbundle}.
\\

\indent Finally, consider the connection $A_\omega: TM\rightarrow \mathbb{T}M$, given by
\begin{align}\label{2 connections}
     A_\omega&: X\mapsto X \oplus \iota_X\omega,
\end{align}
where $\omega = g(\cdot , J_{TM} \cdot)$ is the fundamental $2$-form. One can check that the image of $A_{\omega}$ is isotropic (as it is the graph of a 2-form), and that it gives the identity map when composed with the projection onto $TM$. Then $A_\omega$ is a para-complex connection for the following reason. We note that $g(X^{\pm},X^{\pm}) = 0$ as the eigenbundles $T^{\pm}M$ are isotropic, and so $\iota_{X^{+}}\omega \in \Omega^{(0,1)}(M)$ and $\iota_{X^{-}}\omega \in \Omega^{(1,0)}(M)$, which ensures that $A_{\omega}$ is a para-complex connection, as it maps the $\pm1$-eigenbundles of $TM$ into the $\pm1$-eigenbundles of $\mathbb{T}M$.
\\

\indent The closure of the image of $A_\omega$ under the Courant bracket corresponds to $Gr_\omega\subset \mathbb{T}M$ being a Dirac subbundle, and hence corresponds to $\omega$ being closed. Therefore the K\"{a}hler condition is equivalent to the the para-complex connection $A_\omega$ being flat. Additionally, one can easily check that if $A_\sigma$ is a connection generated by the $2$-form $\sigma$, then $A$ is para-complex if and only if $\sigma(J\cdot,J\cdot)= -\sigma$ (i.e. $\sigma$ is compatible with $J$). We can see that the interesting facts about para-Hermitian manifolds can be translated nicely to the setting of para-holomorphic algebroids and para-complex connections. This construction is due to Svoboda \cite{Svoboda}, though we have chosen to take a different perspective. We will now look at some non-trivial examples arising from Lie groups. 
\subsection{ \normalsize Quadratic Lie Groups}
\begin{Example}\label{Para-Hermitian Lie Groups}
A para-Hermitian structure on $\mathbb{T}G$ for any Lie group $G$ with a quadratic Lie algebra $\mathfrak{g}$.
\end{Example}
\indent A class of examples can be derived from the discussion of Dirac structures on Lie groups in \cite{Spinors}. Let $G$ be a Lie group with Lie algebra $\mathfrak{g}$, and suppose that this Lie algebra carries with it an $Ad$-invariant inner product $B$. Let $\theta^L$,$\theta^R\in \Omega^1(G)\times \mathfrak{g}$ be the left/right-Maurer-Cartan forms, respectively. One can define a bi-invariant (with respect to the adjoint action of $G$) pseudo-Riemannian metric on $G$ by $B(\theta^L, \theta^L)$. One can also define the bi-invariant $3$-form $\eta \in \Omega^3(G)$ by 
\begin{align*}
    \eta := \frac{1}{12}B(\theta^L, [\theta^L,\theta^L]_{\mathfrak{g}}).
\end{align*}
The bi-invariance of $\eta$ implies that it is also closed (as its Jacobiator vanishes on constant sections), and so one can define the $\eta$-twisted Courant bracket on $\mathbb{T}G$, $\llbracket \cdot ,\cdot \rrbracket_{\eta}$. Now, let $D = G\times G$ with Lie algebra $\mathfrak{d} = \mathfrak{g}\oplus \mathfrak{g}$. We have a natural smooth action 
\begin{align*}
\mathcal{A}:D\rightarrow \textrm{Diff}(G), \, \, \, \, \, \, \, \, \, \, \mathcal{A}(g,g') = l_{g}\circ r_{g^{-1}},
\end{align*}
for $g,g'\in G$, whose corresponding infinitesimal action 
\begin{align*}
\mathcal{A}_{*} : \mathfrak{d}\rightarrow \mathfrak{X}(G), \, \, \, \, \, \, \, \, \, \, \, \mathcal{A}_{*} (X,X') = X^L-(X')^R,
\end{align*}
where $X^L_g = l_g(X)$ and $(X')^R_g = r_g(X')$ for any $X,X'\in \mathfrak{g}$. The authors then use this action to define a $D$-equivarant map 
\begin{align*}
    s : \mathfrak{d}\rightarrow \Gamma(\mathbb{T}G)&, \, \, \, \, \, \, \, \, \, \,  s(X,X') = s^L(X) + s^R(X'),
    \\
    s^L(X) = X^L \oplus \frac{1}{2}B(\theta^L, X)&, \, \, \, \, \, \, \, \, \, \,  s^R(X') = -(X')^R \oplus \frac{1}{2}B(\theta^R,X').
\end{align*}
Equipping $\mathfrak{d}$ with the bilinear form $B_{\mathfrak{d}}$, given by $+B$ on the first $\mathfrak{g}$ summand of $\mathfrak{d} = \mathfrak{g}\oplus \mathfrak{g}$, and $-B$ on the second summand gives a split signature bilinear form (exactly what it needed for $\mathfrak{d}$ to admit a para-complex structure). The map $s:\mathfrak{d}\rightarrow \Gamma(\mathbb{T}G)$ has the following important properties:
\begin{enumerate}
    \item $\langle s(e_1),s(e_2)\rangle = B_{\mathfrak{d}}(e_1,e_2)$,
    \item $\llbracket s(e_1),s(e_2)\rrbracket_{\eta} = s([e_1,e_2])$,
    \item $\Upsilon(s(e_1),s(e_2),s(e_3)) = B_{\mathfrak{d}}(e_1,[e_2,e_3])$,
\end{enumerate}
where $e_i\in \mathfrak{d}$, and $\Upsilon(x_1,x_2,x_3) = -\langle \llbracket x_1,x_2\rrbracket,x_3\rangle$. These identities tell us that $s$ is a $D$-equivariant isometric isomorphism of vector bundles $G\times \mathfrak{d} = \mathbb{T}G$, identifying the twisted Courant bracket $\llbracket \cdot ,\cdot \rrbracket_{\eta}$ on $\mathbb{T}G$ with the unique Courant bracket on $G\times \mathfrak{d}$ which agrees with the Lie bracket on $\mathfrak{d}$ for constant sections. Note that $s$ is not necessarily a Courant algebroid ismorphism as it does not preserve the anchors. For this reason, any pair of Lagrangian subalgebras on $\mathfrak{d}$ (called a Manin-triple) is equivalent to defining an integrable para-Hermitian algebroid structure on $\mathbb{T}G$ with the $\eta$-twisted bracket, and the map $s$ can be understood as an isomorphism of para-Hermitian algebroids. A classification of Lagrangian subalgebras of Lie algebras in the form $\mathfrak{d} = \mathfrak{g}\oplus \overline{\mathfrak{g}}$ where $\mathfrak{g}$ is a complex semi-simple Lie algebra was completed in \cite{Karolinsky}, and so we have a wealth of examples of para-Hermitian algebroids corresponding to Lie groups with quadratic Lie algebras on the Courant algebroid $\mathbb{T}G$.
\subsection{ \normalsize Structures on $\mathbb{T}G$ induced by the Iwasawa Decomposition}
\begin{Example}\label{Cool Example}
A example of an exact para-holomorphic algebroid with a compatible flat para-complex connection over $\mathbb{T}G$ where $G$ is a Lie group with a quadratic, semi-simple Lie algebra $\mathfrak{g}$ corresponding to the Cartan-Dirac structure.
\end{Example}

\indent Example \ref{Para-Hermitian Lie Groups} can be pushed further to yield an exact para-holomorphic algebroid with a para-complex connection. This construction is based on the Iwasawa deocomposition, which one can see described in \cite{Jana}. We will first give the highlights of the Iwasawa decomposition. Let $\mathfrak{g}$ be a complex semi-simple Lie algebra. Then $\mathfrak{g}$ is guaranteed to have a compact real form $\mathfrak{k}$, with respect to which $\mathfrak{g}$ (when viewed as a real vector space) admits the decomposition $\mathfrak{g} = \mathfrak{k}\oplus i \mathfrak{k}$. Since $\mathfrak{k}$ is a compact Lie algebra, the Killing form $\kappa$ on $\mathfrak{k}$, defined by $\kappa(X,Y) = 2\dim(\mathfrak{k})Tr(ad_{X}ad_{Y})$, is negative definite. We can extend the Killing form on $\mathfrak{k}$ to all of $\mathfrak{g}$ by $\mathbb{C}$-linearity. One then obtains a non-degenerate symmetric bilinear form on $\mathfrak{g}$ by taking $-Im(\kappa)$, which makes $\mathfrak{g}$ into a quadratic Lie algebra, for which $\mathfrak{k}$ is a Lagrangian subalgebra.
\\

\indent There exists on $\mathfrak{g}$ another Lagrangian subalgebra that is naturally dual to $\mathfrak{k}$. To begin, pick a maximal Cartan subalgebra $\mathfrak{t}\subset \mathfrak{k}$, and note that $\mathfrak{a} = i\mathfrak{t}$ is also a commuting subalgebra of $\mathfrak{g}$. If we consider the root space decomposition with respect to this choice of $\mathfrak{a}$ and define a notion of positivity using the inner product on $\mathfrak{g}$, then we can define $\mathfrak{n}= \bigoplus_{\Sigma^{+}}g_{\lambda}$ to be the direct sum over the positive roots spaces. Then $\mathfrak{g}$, when viewed as a real Lie algebra, admits the decomposition into the sum of Lagrangian subalgebras $\mathfrak{g} = \mathfrak{k} \oplus (\mathfrak{a}\oplus \mathfrak{n})$, which we refer to as the Iwasawa decomposition of $\mathfrak{g}$. 
\\

\indent Returning to Example \ref{Para-Hermitian Lie Groups}, let $G$ be a Lie group with complex semi-simple Lie algebra $\mathfrak{g}$. The anchor map on $G\times \mathfrak{d} = G\times (\mathfrak{g}\oplus \overline{\mathfrak{g}})$ is simply given by projection onto the first summand, where $TG \cong G\times \mathfrak{g}$ under left trivialization. One then defines an integrable para-Hermitian structure $J_{\mathfrak{g}}$ on $TG$ by taking the Iwasawa decomposition $\mathfrak{g} = \mathfrak{k}\oplus (\mathfrak{a}\oplus\mathfrak{n})$, and defining $\mathfrak{k}$ to be the $+1$ eigenbundle, and $(\mathfrak{a}\oplus \mathfrak{n})$ to be the $-1$-eigenbundle. This lifts to a natural para-Hermitian structure on $G\times \mathfrak{d}$ by $J_\mathfrak{d}(X\oplus Y) = J_{\mathfrak{g}}(X) \oplus J_{\mathfrak{g}}(Y)$. The most natural connection on $\mathbb{T}G$ is the diagonal map $A_{\Delta}: TG \rightarrow G\times \mathfrak{d}$, given by $A_{\Delta}(g,X)= s(g,X,X)$ for $(g,X)\in T_{g}G$. The connection $A_{\Delta}$ is flat with respect to the induced Courant bracket on $G\times \mathfrak{g}$. The map $s\circ A_{\Delta}$ also preserves $\eta$-twisted bracket, and is often referred to as the Cartan-Dirac structure. The Dirac foliation of $s\circ A_{\Delta}$ corresponds to the conjugacy classes of $G$.
\\

\indent It's quite easy to see that the connection $A_\Delta$ is para-complex and that the anchor commutes with the para-complex structure on $G\times \mathfrak{d}$, so $G\times \mathfrak{d}$ is an example of an exact para-holomorphic algebroid with a flat para-complex connection. Further, the Cartan-Dirac structure on $\mathbb{T}G$ can be realized as the image of this para-complex connection under the para-Hermitian algebroid isomorphism $s$, which has the form
\begin{align}\label{Poisson precurser}
    s\circ A_{\Delta}(g,X) &= \left( X^L- X^R, \frac{1}{2} B(\theta^L + \theta^R, X) \right).
\end{align}
It is instructive that $\mathbb{T}G$ and $G\times \mathfrak{d}$ are isomorphic as para-holomorphic algebroids under $s$ (as $s$ covers the identity map on $G$ which is a bi-para-holomorphic isometric diffeomorphism), but not as Courant algebroids. This illustrated the utility of considering para-holomorphic algebroids in the first place. As a concrete example, consider the Lie group $SL(n;\mathbb{C})$. The Lie group $SL(n;\mathbb{C})$ admits an Iwasawa decomposition with $\mathfrak{su}(n)$ as the choice of compact real form, and so $K = SU(n)$,
\begin{align*}
    A &= \left\{ \begin{bmatrix}
    t_1&0&\cdots &0
    \\
    0 & t_2 & \cdots & 0 \\
    \vdots & \vdots & \ddots &\vdots \\
    0 & 0 & \cdots & t_n
    \end{bmatrix}, \, \, \, \, \, \, t_i \in \mathbb{R}, \, \, \, \, \, \prod_i t_i = 1\right\} , 
    \\
    N&= \left\{ \begin{bmatrix}
    1 & \theta_{12} & \cdots & \theta_{1n}\\
    0 & 1 & \cdots & \theta_{2n}\\
    \vdots & \vdots & \ddots &\vdots \\
    0 & 0 &\cdots & 1
    \end{bmatrix},  \, \, \, \, \, \theta_i\in \mathbb{R}, \textrm{ not all }\theta_i=0 \right\}.
\end{align*}
Therefore $\mathbb{T}SL(n;\mathbb{C})$ with the $\eta$-twisted Courant bracket admits an exact para-holomorphic algebroid structure with an associated flat para-complex connection (induced by the diagonal inclusion of $\mathfrak{sl}(n;\mathbb{C})$ into its double) for every $n$. 
\subsection{Poisson-Lie Groups}
\indent The connection to Poisson-Lie groups here is exciting. In the above example, all that was required was that the complex semi-simple Lie algebra $\mathfrak{g}$ has a compact real form. The non-degenerate symmetric bilinear form could be derived from the Killing form on the subalgebra corresponding to that compact real form. Poisson-Lie groups allow us to begin with a Lie group $K$ together with a multiplicative Poisson structure $\Pi_K$ and then find an exact para-holomorphic algebroid structure with a flat para-complex connection on the Drinfeld double $D$. This was first done by Lu in \cite{Lu1}. Consider the definition:
\begin{Def}
Let $K$ be a Lie group and $\Pi_K$ a Poisson structure on $K$. We say that $\Pi_K$ is multiplicative if the multiplication map $m:K\times K \rightarrow K$ is a Poisson map, where $K\times K$ is endowed with the product Poisson structure. For $k_1,k_2\in K$, the multiplicativity condition is equivalent to 
\begin{align*}
    (\Pi_{K})_{k_1k_2} = L_{k_1}(\Pi_{K})_{k_2} + R_{k_2}(\Pi_K)_{k_1}.
\end{align*}
\end{Def}
\indent An immediate consequence of this definition is that all multiplicative Poisson structures vanish at the identity. The Poisson structure therefore admits a linearization at the identity to give a map $d_e\Pi_K : \mathfrak{k}\rightarrow \mathfrak{k}\wedge \mathfrak{k}$. The condition that the dual map $(d_e\Pi_K)^{*}:\mathfrak{k}^{*}\wedge \mathfrak{k}^{*}\rightarrow \mathfrak{k}^{*}$ must satisfy in order to be a Lie bracket is the Jacobi identity, which is equivalent to the fact that the bivector $\Pi_K$ is Poisson. This means that $\mathfrak{k}^{*}$ is a Lie algebra in its own right. Lu was able to show that there is a unique connected simply-connected Poisson Lie group $(K^{*},\Pi_{K^{*}})$ such that $(d_e\Pi_{k^{*}})^{*}$ is the Lie bracket on $\mathfrak{k}$. This gives $(\mathfrak{k}\oplus \mathfrak{k}^{*})$ the structure of a Lie bialgebra. Finally, one considers the double Lie algebra to the pair $(\mathfrak{k},\mathfrak{k}^{*})$ to be $\mathfrak{d} = \mathfrak{k}\oplus \mathfrak{k}^{*}$ together with the Lie bracket
\begin{align*}
    [X\oplus \xi , Y\oplus \eta]_{\mathfrak{d}} &= \left([X,Y]_{\mathfrak{k}} + ad_{\xi}^{*}(Y) - ad_{\eta}^{*}(X)\right) \oplus \left( [\xi,\eta]_{\mathfrak{g}^{*}} + ad_{X}^{*}(\eta)- ad_{Y}^{*}(\xi) \right),
\end{align*}
for $X,Y\in \mathfrak{k}$, $\xi,\eta\in \mathfrak{k}^{*}$.
\\

\indent Now, one can consider the unique connected simply-connected Lie group $D$ corresponding to the Lie algebra $\mathfrak{d}$. It is a fact about Lie bialgebras that the natural pairing $\langle X\oplus \xi ,Y\oplus \eta\rangle =  \langle \xi , Y\rangle + \langle \eta , X\rangle$ is $ad$-invariant with respect to the above bracket, and so $\mathfrak{d}$ is a quadratic Lie algebra with $\mathfrak{k}$ and $\mathfrak{k}^{*}$ as dual Lagrangian subalgebras \cite{Poisson Structures}. There exists two natural inclusions $\mathfrak{k},\mathfrak{k}^{*} \hookrightarrow \mathfrak{d}$. These inclusions integrate to two local diffeomorphisms 
\begin{align}\label{2Diffeo's}
    \Psi &: K\times K^{*}\rightarrow D, & \Phi &: K^{*}\times K \rightarrow D,
    \\
    \nonumber &(k,u)\mapsto ku & &(u,k)\mapsto uk
\end{align}
\indent The first of these local diffeomorphisms defines a para-complex structure on $D$, which is clearly compatible with the quadratic Lie algebra structure. Further, given a local basis $\{e_i\}$ of $\mathfrak{k}$ and corresponding dual basis $\{\varepsilon^i\}$ of $\mathfrak{k}^{*}$, we can define the classical $R$-matrix
\begin{align}\label{R-Matrix}
    \Lambda &= \frac{1}{2}\sum_{i}e_i\wedge \varepsilon^i.
\end{align}
This classical $R$-matrix defines a natural Poisson structure on $D$ given by
\begin{align*}
    \Pi_D^{+} &= \Lambda^L + \Lambda^R.
\end{align*}
\indent The Poisson structure $\Pi_D^{+}$ falls short of defining a connection on $\mathbb{T}D$ because it is possible that it is degenerate, and so will not compose with the anchor to give the identity map. There is a condition under which $\Pi_D^{+}$ is non-degenerate however, and it is related to the two diffeomorphisms in equation \eqref{2Diffeo's}. Following Lu, we define the left infinitesimal Dressing vector field on $K$ (resp. $K^{*}$) induced by the element $\xi \in \mathfrak{k}^{*}$ (resp $X\in \mathfrak{k}$) by 
\begin{align*}
    \lambda &: \mathfrak{k}^{*}\rightarrow \mathfrak{X}(K), & \rho &: \mathfrak{k}\rightarrow \mathfrak{X}(K^{*})
    \\
     &\xi\mapsto \tilde{\Pi}_K(\xi^L) & &X\mapsto \tilde{\Pi}_{K^{*}}(X^L)
\end{align*}
where $\tilde{\Pi}_K:T^{*}K\rightarrow TK$ is the bundle map induced by $\Pi_K$, and $\xi^L$ is the left-invariant $1$-form induced by $\xi$ on $K$. One can do something similar for the dual Lie group $K^{*}$ to define the Dressing action of $\mathfrak{k}$ on $K^{*}$. We say that $(K,\Pi_K)$ is complete if the left infinitesimal Dressing action integrates to an action of $K^{*}$ on $K$ (i.e. when the flows of the vector fields on $K$ are complete). Finally, from Proposition 2.45 in \cite{Lu1}, we have the following result.
\begin{Thm}
\cite{Lu1} Assume that $(K,\Pi_K)$ is a complete and simply connected Poisson-Lie group. Then the Poisson structure $(D,\Pi_D^{+})$ is non-degenerate (and therefore symplectic) everywhere.
\end{Thm}
Relating back to the maps in equation \eqref{2Diffeo's}, Lu also tells us the following:
\begin{Thm}
\cite{Lu1} A simply-connected Poisson-Lie group $(K,\Pi_K)$ is complete if and only if the maps $\Psi : K\times K^{*}\rightarrow D$ and $\Phi:K^{*}\times K \rightarrow D$ are diffeomorphisms.
\end{Thm}
For this reason, we restrict our attention to complete simply-connected Poisson Lie groups $(K,\Pi_K)$. If $(K,\Pi_K)$ is a complete simply-connected Poisson-Lie group, then the Drinfeld double admits a global diffeomorphism $D\cong K\times K^{*}$, and an invertible Poisson structure $\Pi_D$. If we let the related symplectic structure by $\omega_D$, we know that $\omega_D$ is closed because $\Pi_D^{+}$ is Poisson, and so $\omega_D$ is symplectic on $D$. We can now begin the construction of a para-holomorphic algebroid with a flat para-complex connection in a similar way to Example \ref{Normal}.
\\

\indent Consider the standard Courant algebroid on $\mathbb{T}D$ induced by the Lie bracket on $\mathfrak{d} = \mathfrak{k}\oplus \mathfrak{k}^{*}$. We have a connection $A_{\omega_D}$, given by the graph of $\omega_D$. There is a natural para-complex structure $J_D$ on $TD$ coming from the span of the left invariant vector fields on $K$ and $K^{*}$, making $\{(e_i)^{L},(\varepsilon^i)^L\}$ a global frame for $TD$ . By the non-degeneracy of $\omega_D$, we have a global frame for $\mathbb{T}D$, given by $\{(e_i)^L,(\varepsilon^i)^L, \,  \tilde{\omega}_D((e_i)^L), \, \tilde{\omega}_D ((\varepsilon^i)^L) \}$. If we chose to extend $J_{D}$ to an almost para-complex structure $\tilde{J}_D$ on $\mathbb{T}D$ by identifying the $+1$-eigenbundle as the $C^{\infty}(D)$-span of $\{(e_i)^L, \tilde{\omega}_D((e_i)^L)\}$, and the $-1$-eigenbundle as the $C^{\infty}(D)$-span of $\{(\varepsilon^i)^L, \tilde{\omega}_D((\varepsilon^i)^L)\}$, then it is clear that with respect to this almost para-complex structure, $A_{\omega_{D}}$ defines a para-complex connection. 
\\

\indent To summarize, we have an almost para-complex structure $\tilde{J}_D$ with a para-complex connection $A_{\omega_D}$ on $\mathbb{T}D$. The induced para-Hermitian structure then comes from extending the natural left invariant pairing $\langle \theta^L, \theta^L\rangle$ to $\mathbb{T}D$ by taking $\langle\theta^L\circ \tilde{\Pi}_D,\theta^L\circ \tilde{\Pi_D}\rangle$ on the second factor, which we denote simply by $\langle \cdot ,\cdot \rangle_{L,\Pi_D}$. This metric is clearly para-Hermitian with respect to $\tilde{J}_D$, and the anchor $\rho_{TD}$ given by projection onto the first factor is clearly para-holomorphic. 
\\

\indent This construction uses the standard Courant algebroid structure and so by Example \ref{Normal} both the eigenbundles are integrable. Therefore, $(\mathbb{T}D, \tilde{J}_D, \langle \cdot ,\cdot \rangle_{L,\Pi_D}, \rho_{TD})$ is an exact para-holomorphic algebroid with a para-complex connection $A_{\omega_{D}}$, and this is an example that can be done over any complete semi-simple Poisson-Lie group that is different from the one induced by the Cartan-Dirac structure previously, and uses the standard Courant algebroid structure on $\mathbb{T}D$. 
\pagebreak

\end{document}